\documentclass[11pt]{amsart}
\usepackage[margin=1in]{geometry}
\usepackage{latexsym}
\usepackage{amsfonts}
\usepackage{amsmath}
\usepackage{amssymb}
\usepackage{amsthm}
\usepackage{enumerate}
\usepackage[hang,flushmargin]{footmisc}
\usepackage{caption}
\usepackage{tabu}
\usepackage{mathrsfs}
\usepackage{amsaddr}
\usepackage{subfig}

\newcommand{\Comp}{\operatorname{Comp}}
\newcommand{\Av}{\operatorname{Av}}

\usepackage{graphicx}
\usepackage{epstopdf}
\usepackage{epsfig}
\usepackage{caption}
 
\usepackage{bm} 
\usepackage{cite}

\makeatletter
\newtheorem*{rep@theorem}{\rep@title}
\newcommand{\newreptheorem}[2]{%
\newenvironment{rep#1}[1]{%
 \def\rep@title{#2 \ref{##1}}%
 \begin{rep@theorem}}%
 {\end{rep@theorem}}}
\makeatother

\newtheorem{theorem}{Theorem}[section]
\newreptheorem{theorem}{Conjecture}

\newtheorem{corollary}{Corollary}[section]

\newtheorem{conjecture}{Conjecture}[section]
\newtheorem{question}{Question}[section]

\theoremstyle{definition}
\newtheorem{definition}{Definition}[section]

\newtheorem{example}{Example}[section]

\begin{document}
\title{Stack-Sorting Preimages of Permutation Classes}
\author{Colin Defant}
\address{Princeton University \\ Fine Hall, 304 Washington Rd. \\ Princeton, NJ 08544}
\email{cdefant@princeton.edu}

\begin{abstract}
We extend and generalize many of the enumerative results concerning West's stack-sorting map $s$. First, we prove a useful theorem that allows one to efficiently compute $|s^{-1}(\pi)|$ for any permutation $\pi$, answering a question of Bousquet-M\'elou. We then enumerate permutations in various sets of the form $s^{-1}(\Av(\tau^{(1)},\ldots,\tau^{(r)}))$, where $\Av(\tau^{(1)},\ldots,\tau^{(r)})$ is the set of permutations avoiding the patterns $\tau^{(1)},\ldots,\tau^{(r)}$. These preimage sets often turn out to be permutation classes themselves, so the current paper represents a new approach, based on the theory of valid hook configurations, for solving classical enumerative problems. In one case, we solve a problem previously posed by Bruner. We are often able to refine our counts by enumerating these permutations according to their number of descents or peaks. Our investigation not only provides several new combinatorial interpretations and identities involving known sequences, but also paves the way for several new enumerative problems. 
\end{abstract}

\maketitle

\bigskip

\section{Introduction}\label{Sec:Intro}

Throughout this paper, we write permutations as words in one-line notation. Let $S_n$ denote the set of permutations of $[n]$. A \emph{descent} of a permutation $\pi=\pi_1\cdots\pi_n\in S_n$ is an index $i\in[n-1]$ such that $\pi_i>\pi_{i+1}$. A \emph{peak} of $\pi$ is an index $i\in\{2,\ldots,n-1\}$ such that $\pi_{i-1}<\pi_i>\pi_{i+1}$. 

\begin{definition}\label{Def1}
We say the permutation $\sigma=\sigma_1\cdots\sigma_n$ \emph{contains the pattern} $\tau=\tau_1\cdots\tau_m$ if there are indices $i_1<\cdots<i_m$ such that $\sigma_{i_1}\cdots\sigma_{i_m}$ has the same relative order as $\tau$. Otherwise, we say $\sigma$ \emph{avoids} $\tau$. Denote by $\Av(\tau^{(1)},\ldots,\tau^{(r)})$ the set of permutations that avoid the patterns $\tau^{(1)},\ldots,\tau^{(r)}$. Let $\Av_n(\tau^{(1)},\ldots,\tau^{(r)})=\Av(\tau^{(1)}\ldots,\tau^{(r)})\cap S_n$. Let $\Av_{n,k}(\tau^{(1)},\ldots,\tau^{(r)})$ be the set of permutations in $\Av_n(\tau^{(1)},\ldots,\tau^{(r)})$ with exactly $k$ descents. 
\end{definition}     

A set of permutations is called a \emph{permutation class} if it is the set of permutations avoiding some (possibly infinite) collection of patterns. Equivalently, a permutation class is a downset in the poset of all permutations ordered by containment. The \emph{basis} of a class is the minimal set of permutations not in the class. 

The notion of pattern avoidance in permutations, which has blossomed into an enormous area of research and which plays a lead role in the present article, began in Knuth's book \emph{The Art of Computer Programming} \cite{Knuth}. In this book, Knuth described a so-called \emph{stack-sorting algorithm}; it was the study of the combinatorial properties of this algorithm that led him to introduce the idea of pattern avoidance. In his 1990 Ph.D. thesis \cite{West}, West defined a deterministic version of Knuth's stack-sorting algorithm, which we call the \emph{stack-sorting map} and denote by $s$. The stack-sorting map is a function defined by the following procedure. 

Suppose we are given an input permutation $\pi\in S_n$. Place this permutation on the right side of a vertical ``stack." Throughout this process, if the next entry in the input permutation is larger than the entry at the top of the stack or if the stack is empty, the next entry in the input permutation is placed at the top of the stack. Otherwise, the entry at the top of the stack is annexed to the end of the growing output permutation. This procedure stops when the output permutation has length $n$. We then define $s(\pi)$ to be this output permutation. Figure \ref{Fig1} illustrates this procedure and shows that $s(3142)=1324$.  

\begin{figure}[t]
\begin{center}
\includegraphics[width=1\linewidth]{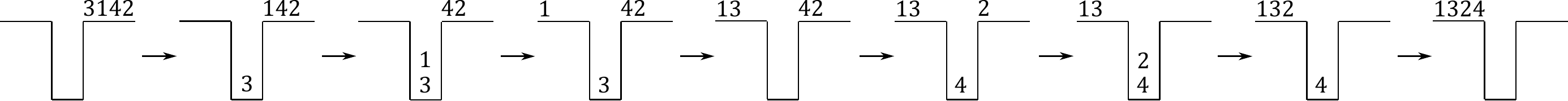}
\end{center}  
\caption{The stack-sorting map $s$ sends $3142$ to $1324$.}
\end{figure}\label{Fig1}

If $\pi\in S_n$, we can write $\pi=LnR$, where $L$ (respectively, $R$) is the (possibly empty) substring of $\pi$ to the left (respectively, right) of the entry $n$. West observed that the stack-sorting map can be defined recursively by $s(\pi)=s(L)s(R)n$ (here, we also have to allow $s$ to take permutations of arbitrary finite sets of positive integers as arguments). There is also a natural definition of the stack-sorting map in terms of tree traversals of decreasing binary plane trees (see \cite{Bona,Defant,Defant2}). 

The ``purpose" of the stack-sorting map is to sort the input permutation into increasing order. Hence, we say a permutation $\pi\in S_n$ is \emph{sortable} if $s(\pi)=123\cdots n$. The above example illustrates that the stack-sorting map does not always do its job. In other words, not all permutations are sortable. In fact, the following characterization of sortable permutations follows from Knuth's work. 

\begin{theorem}[\!\!\cite{Knuth}]\label{Thm1}
A permutation $\pi$ is sortable if and only if it avoids the pattern $231$. 
\end{theorem} 

Even if a permutation is not sortable, we can still try to sort it via iterated use of the stack-sorting map. In what follows, $s^t$ denotes the composition of $s$ with itself $t$ times.  

\begin{definition}
A permutation $\pi\in S_n$ is called $t$-\emph{stack-sortable} if $s^t(\pi)=123\cdots n$. Let $\mathcal W_t(n)$ denote the set of $t$-stack-sortable permutations in $S_n$. Let $W_t(n)=|\mathcal W_t(n)|$. 
\end{definition}

Theorem \ref{Thm1} states that $\mathcal W_1(n)=\Av_n(231)$, so it follows from the well-known enumeration of $231$-avoiding permutations that $W_1(n)=C_n=\frac{1}{n+1}{2n\choose n}$ is the $n^\text{th}$ Catalan number. In his thesis, West proved \cite{West} that a permutation is $2$-stack-sortable if and only if it avoids the pattern $2341$ and also avoids any $3241$ pattern that is not part of a $35241$ pattern. He also conjectured the following theorem, which Zeilberger proved in 1992. 

\begin{theorem}[\!\!\cite{Zeilberger}]\label{Thm2}
We have 
\begin{equation}\label{Eq1}
W_2(n)=\frac{2}{(n+1)(2n+1)}{3n\choose n}.
\end{equation}
\end{theorem}

Combinatorial proofs of this theorem arose later in \cite{Cori,Dulucq2,Dulucq,Goulden}. Some authors have studied the enumerations of $2$-stack-sortable permutations according to certain statistics \cite{BonaSimplicial,Bousquet98,Bouvel,Dulucq}. Recently, the authors of \cite{Duchi} introduced new combinatorial objects known as \emph{fighting fish} and showed that they are counted by the numbers $W_2(n)$. Fang has now given a bijection between fighting fish and $2$-stack-sortable permutations \cite{Fang}. The authors of \cite{Bevan} study what they call $n$-\emph{point dominoes}, and they have made the fascinating discovery that the number of these objects is $W_2(n+1)$. In \cite{Bona}, B\'ona draws attention to the fact that a certain class of lattice paths are counted by the numbers $2^{2n-1}W_2(n)$ (see \cite{Bousquet02}), and he asks about the possibility of using this result to obtain a simple combinatorial proof of \eqref{Eq1}. 

The primary purpose of this article is to study preimages of permutation classes under the stack-sorting map. This is a natural generalization of the study of sortable and $2$-stack-sortable permutations since $\mathcal W_1(n)=s^{-1}(\Av_n(21))$ and $\mathcal W_2(n)=s^{-1}(\Av_n(231))$. In fact, Bouvel and Guibert \cite{Bouvel} have already considered stack-sorting preimages of certain classes in their study of permutations that are sortable via multiple stacks and $D_8$ symmetries (we state some of their results in Section \ref{Sec:Back}). Claesson and \'Ulfarsson \cite{Claesson} have also studied this problem in relation to a generalization of classical permutation patterns known as \emph{mesh patterns}, which were introduced\footnote{The idea of a mesh pattern appears earlier under different names. For example, it appears in \cite{Bousquet06,Dulucq,West}. Mesh patterns are sometimes called \emph{barred patterns}.} in \cite{Branden}. They showed that each set of the form $s^{-1}(\Av(\tau^{(1)},\ldots,\tau^{(r)}))$ can be described as the set of permutations avoiding a specific collection of mesh patterns, and they provided an algorithm for computing this collection. From this point of view, the current paper provides a new method for counting permutations that avoid certain mesh patterns. In specific cases, $s^{-1}(\Av(\tau^{(1)},\ldots,\tau^{(r)}))$ is a bona fide permutation class. For example, $s^{-1}(\Av(m(m-1)\cdots 321))$ is a permutation class. Hence, we give a new method for enumerating (or at least estimating) some permutation classes. 

The idea to count the preimages of a permutation under the stack-sorting map  dates back to West, who called $|s^{-1}(\pi)|$ the \emph{fertility} of the permutation $\pi$ and went to great lengths to compute the fertilities of the permutations of the forms \[23\cdots k1(k+1)\cdots n,\quad 12\cdots(k-2)k(k-1)(k+1)\cdots n,\quad\text{and}\quad k12\cdots(k-1)(k+1)\cdots n.\] The very specific forms of these permutations indicates the initial difficulty of computing fertilities. We define the fertility of a set of permutations to be the sum of the fertilities of the permutations in that set. With this terminology, our main goal in this paper is to compute the fertilities of sets of the form $\Av_n(\tau^{(1)},\ldots,\tau^{(r)})$. 

Bousquet-M\'elou \cite{Bousquet} studied permutations with positive fertilities, which she termed \emph{sorted} permutations. She mentioned that it would be interesting to find a method for computing the fertility of any given permutation. This was achieved in much greater generality in \cite{Defant} using new combinatorial objects called ``valid hook configurations." The theory of valid hook configurations was the key ingredient used in \cite{Defant2} in order to improve the best-known upper bounds for $W_3(n)$ and $W_4(n)$, and it will be one of our primary tools in subsequent sections. 

It is important to note that we do not have a thorough understanding of valid hook configurations. This is why an explicit formula for $W_3(n)$ remains out of reach, and it is why enumerating the permutations in $s^{-1}(\Av(\tau^{(1)},\ldots,\tau^{(r)}))$ is still highly nontrivial for many choices of patterns $\tau^{(1)},\ldots,\tau^{(r)}$. However, valid hook configurations still provide a powerful tool for computing fertilities. Recently, the authors of \cite{Defant3} used valid hook configurations to establish connections among permutations with fertility $1$, certain weighted set partitions, and cumulants arising in free probability theory. The current author has also investigated which numbers arise as the fertilities of permutations \cite{Defant4}. 

In Section \ref{Sec:VHCs}, we review the definitions and necessary theorems concerning valid hook configurations. We also prove a theorem that ameliorates the computation of fertilities in many cases. This theorem was stated in \cite{Defant2}, but the proof was omitted because the result was not needed in that paper. We have decided to prove the result here because we will make use of it in our computations. This result is also used in \cite{Defant4}, so it is important that a proof appears in the literature. 

Section \ref{Sec:Back} reviews some facts about generalized patterns and stack-sorting preimages of permutation classes. In Section \ref{Sec:4patterns}, we study the set $s^{-1}(\Av(132,231,312,321))$, which is a permutation class. In Section \ref{Sec5}, we study $s^{-1}(\Av(132,231,321))$ and $s^{-1}(\Av(132,312,321))$, the latter of which is a permutation class. We show that these sets are both enumerated by central binomial coefficients. A corollary of the results in this section actually settles a problem of Bruner \cite{Bruner}. In Section \ref{Sec6}, we consider $s^{-1}(\Av(231,312,321))$, which turns out to be a permutation class. We enumerate this class both directly and by using valid hook configurations, leading to a new identity involving well-studied orderings on integer compositions and integer partitions. Section \ref{Sec7} considers the set $s^{-1}(\Av(132,231,312))$. Finding the fertilities of permutations in $\Av(132,231,312)$ allows us to prove that some of the estimates used in \cite{Defant2} are sharp. In addition, we will find that the permutations in $s^{-1}(\Av(132,231,312))$ are enumerated by the Fine numbers, giving a new interpretation for this well-studied sequence. Section \ref{Sec:312,321} is brief and is merely intended to state that $s^{-1}(\Av(312,321))$ is the permutation class $\Av(3412,3421)$, which Kremer \cite{Kremer} has proven is enumerated by the large Schr\"oder numbers. Section \ref{Sec:132,321} enumerates the permutations in $s^{-1}(\Av(132,321))$. In Section \ref{Sec:Pair}, we prove that $|s^{-1}(\Av_n(132,312))|=|s^{-1}(\Av_n(231,312))|$. Finally, we prove that \[8.4199\leq\lim_{n\to\infty}|s^{-1}(\Av_n(321))|^{1/n}\leq 11.6569\] in Section \ref{Sec:321}. This is notable because $s^{-1}(\Av(321))$ is a permutation class; it is not clear how to obtain such estimates using standard methods for enumerating permutation classes. In most of these sections, we actually refine our counts by enumerating stack-sorting preimages of permutation classes according to the number of descents and according to the number of peaks. 

Our results lead to several fascinating open problems and conjectures. In fact, we believe that this paper opens the door to a vast new collection of enumerative problems. We accumulate these problems and conjectures in Section 11. 

\section{Valid Hook Configurations and Valid Compositions}\label{Sec:VHCs}

In this section, we review some of the theory of valid hook configurations. Our presentation is almost identical to that given in \cite{Defant4}, but we include it here for completeness. Note that the valid hook configurations defined below are, strictly speaking, different from those defined in \cite{Defant} and \cite{Defant2}. For a lengthier discussion of this distinction, see \cite{Defant3}. 

To construct a valid hook configuration, begin by choosing a permutation $\pi=\pi_1\cdots\pi_n\in S_n$. Recall that a descent of $\pi$ is an index $i$ such that $\pi_i>\pi_{i+1}$. Let $d_1<\cdots<d_k$ be the descents of $\pi$. We use the example permutation $3142567$ to illustrate the construction. The \emph{plot} of $\pi$ is the graph displaying the points $(i,\pi_i)$ for $1\leq i\leq n$. The left image in Figure \ref{Fig2} shows the plot of our example permutation. A point $(i,\pi_i)$ is a \emph{descent top} if $i$ is a descent. The descent tops in our example are $(1,3)$ and $(3,4)$. 

\begin{figure}[t]
  \centering
  \subfloat[]{\includegraphics[width=0.2\textwidth]{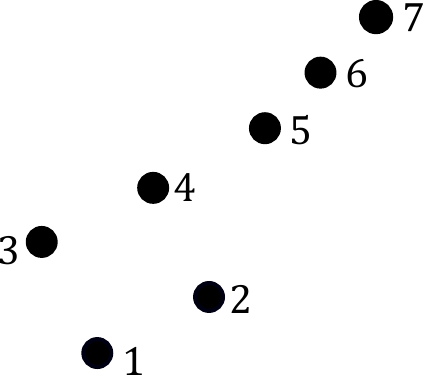}}
  \hspace{1.5cm}
  \subfloat[]{\includegraphics[width=0.2\textwidth]{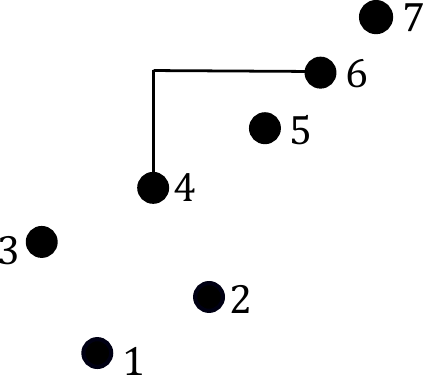}}
  \caption{The left image is the plot of $3142567$. The right images shows this plot along with a single hook.}\label{Fig2}
\end{figure}

A \emph{hook} of $\pi$ is drawn by starting at a point $(i,\pi_i)$ in the plot of $\pi$, moving vertically upward, and then moving to the right until reaching another point $(j,\pi_j)$. In order for this to make sense, we must have $i<j$ and $\pi_i<\pi_j$. The point $(i,\pi_i)$ is called the \emph{southwest endpoint} of the hook, while $(j,\pi_j)$ is called the \emph{northeast endpoint}. The right image in Figure \ref{Fig2} shows our example permutation with a hook that has southwest endpoint $(3,4)$ and northeast endpoint $(6,6)$.

\begin{definition}\label{Def2}
A \emph{valid hook configuration} of $\pi$ is a configuration of hooks drawn on the plot of $\pi$ subject to the following constraints: 

\begin{enumerate}[1.]
\item The southwest endpoints of the hooks are precisely the descent tops of the permutation. 

\item A point in the plot cannot lie directly above a hook. 

\item Hooks cannot intersect each other except in the case that the northeast endpoint of one hook is the southwest endpoint of the other. 
\end{enumerate}  
\end{definition}

\begin{figure}[t]
\begin{center}
\includegraphics[width=.7\linewidth]{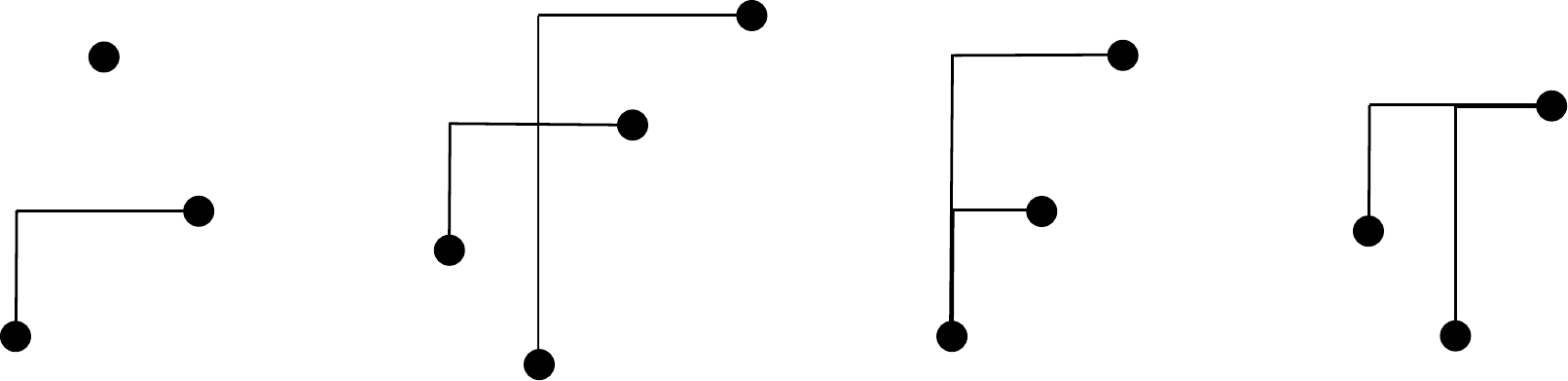}
\caption{Four configurations of hooks that are forbidden in a valid hook configuration.}
\label{Fig3}
\end{center}  
\end{figure}

\begin{figure}[t]
\begin{center}
\includegraphics[width=.7\linewidth]{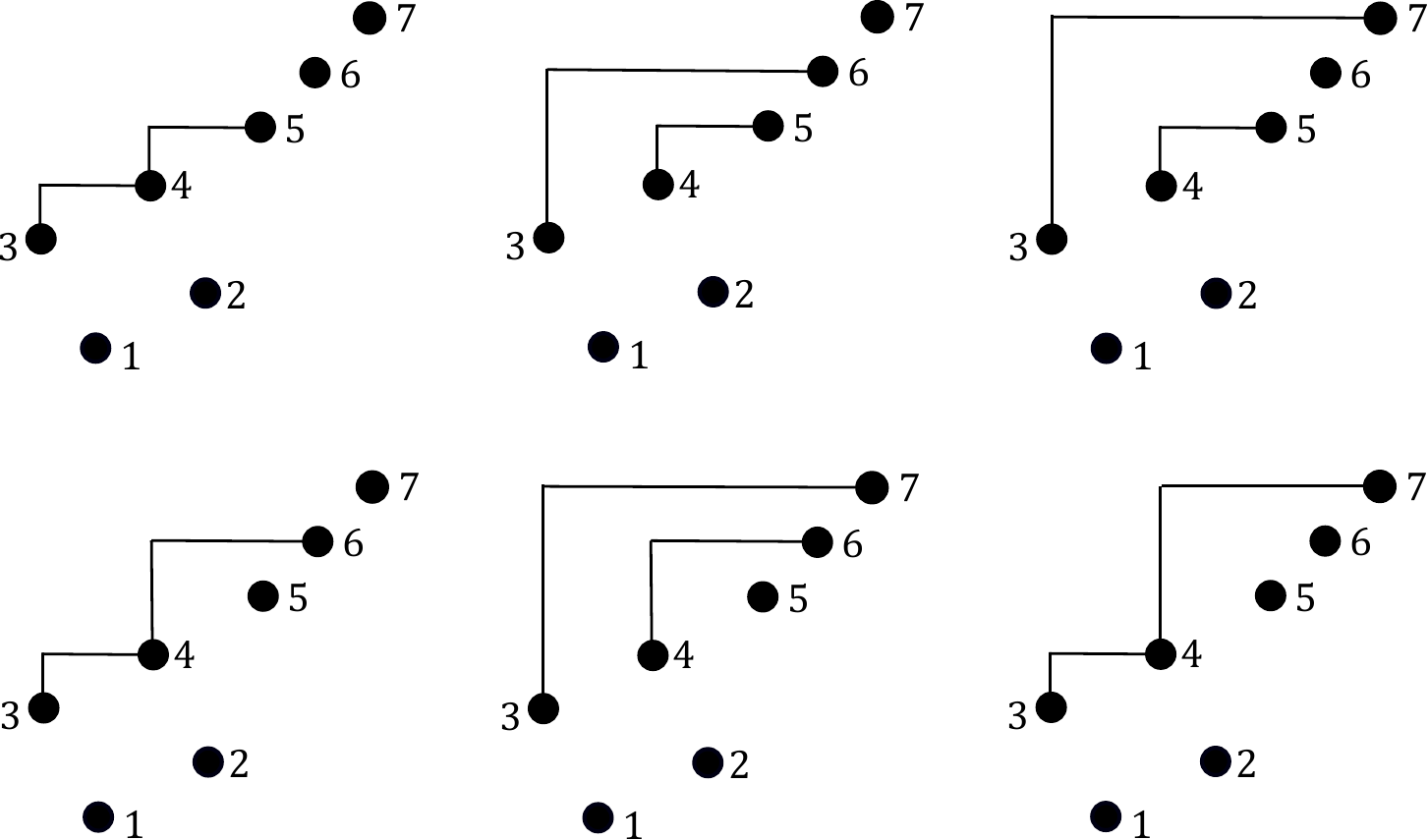}
\caption{All of the valid hook configurations of $3142567$.}
\label{Fig4}
\end{center}  
\end{figure}

Figure \ref{Fig3} shows four placements of hooks that are forbidden by Conditions 2 and 3.
Figure \ref{Fig4} shows all of the valid hook configurations of $3142567$. 
Observe that the total number of hooks in a valid hook configuration of $\pi$ is exactly $k$, the number of descents of $\pi$. Because the southwest endpoints of the hooks are the points $(d_i,\pi_{d_i})$, there a natural ordering of the hooks. Namely, the $i^\text{th}$ hook is the hook whose southwest endpoint is $(d_i,\pi_{d_i})$. We can write a valid hook configuration of $\pi$ concisely as a $k$-tuple $\mathcal H=(H_1,\ldots,H_k)$, where $H_i$ is the $i^\text{th}$ hook. 

A valid hook configuration of $\pi$ induces a coloring of the plot of $\pi$. To color the plot, draw a ``sky" over the entire diagram and color the sky blue. Assign arbitrary distinct colors other than blue to the $k$ hooks in the valid hook configuration. There are $k$ northeast endpoints of hooks, and these points remain uncolored. However, all of the other $n-k$ points will be colored. In order to decide how to color a point $(i,\pi_i)$ that is not a northeast endpoint, imagine that this point looks directly upward. If this point sees a hook when looking upward, it receives the same color as the hook that it sees. If the point does not see a hook, it must see the sky, so it receives the color blue. However, if $(i,\pi_i)$ is the southwest endpoint of a hook, then it must look around (on the left side of) the vertical part of that hook. See Figure \ref{Fig5} for the colorings induced by the valid hook configurations in Figure \ref{Fig4}. Note that the leftmost point $(1,3)$ is blue in each of these colorings because this point looks around the first (red) hook and sees the sky. 

To summarize, we started with a permutation $\pi$ with exactly $k$ descents. We chose a valid hook configuration of $\pi$ by drawing $k$ hooks according to Conditions 1, 2, and 3 in Definition \ref{Def2}. This valid hook configuration then induced a coloring of the plot of $\pi$. Specifically, $n-k$ points were colored, and $k+1$ colors were used (one for each hook and one for the sky). Let $q_i$ be the number of points colored the same color as the $i^\text{th}$ hook, and let $q_0$ be the number of points colored blue (the color of the sky). Then $(q_0,\ldots,q_k)$ is a composition\footnote{Throughout this paper, a \emph{composition of }$b$ \emph{into} $a$ \emph{parts} is an $a$-tuple of \emph{positive} integers whose sum is $b$.} of $n-k$ into $k+1$ parts; we say the valid hook configuration \emph{induces} this composition. Let $\mathcal V(\pi)$ be the set of compositions induced by valid hook configurations of $\pi$. We call the elements of $\mathcal V(\pi)$ the \emph{valid compositions} of $\pi$.  

\begin{figure}[t]
\begin{center}
\includegraphics[width=.7\linewidth]{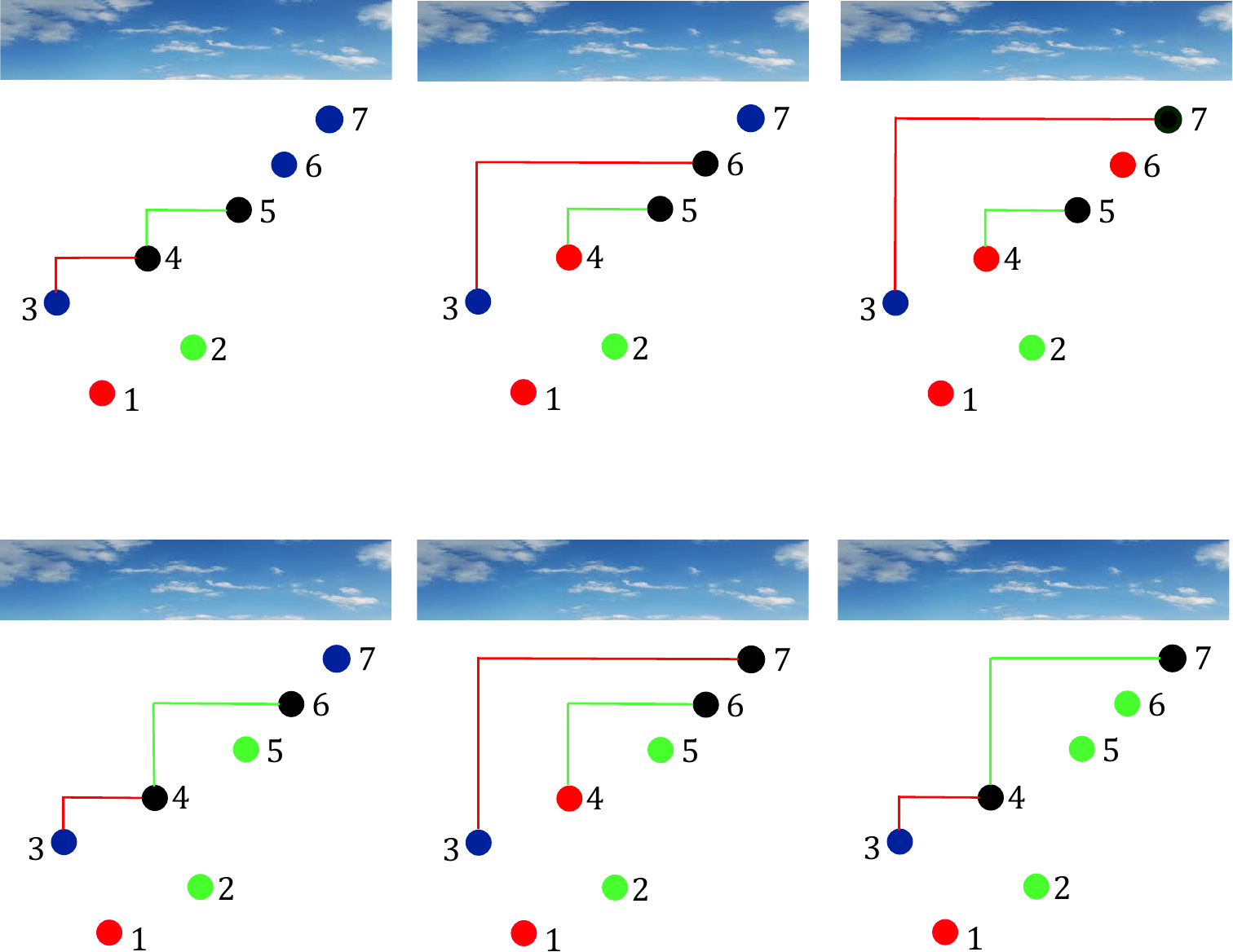}
\caption{The different colorings induced by the valid hook configurations of $3142567$.}
\label{Fig5}
\end{center}  
\end{figure}

Let $C_j=\frac{1}{j+1}{2j\choose j}$ denote the $j^\text{th}$ Catalan number. Given a composition $(q_0,\ldots,q_k)$, let \[C_{(q_0,\ldots,q_k)}=\prod_{t=0}^kC_{q_t}.\] The following theorem explains why valid hook configurations are so useful when studying the stack-sorting map. 

\begin{theorem}[\!\!\cite{Defant}]\label{Thm5}
If $\pi\in S_n$ has exactly $k$ descents, then the fertility of $\pi$ is given by the formula \[|s^{-1}(\pi)|=\sum_{(q_0,\ldots,q_k)\in\mathcal V(\pi)}C_{(q_0,\ldots,q_k)}.\] 
\end{theorem}

\begin{example}\label{Exam1}
The permutation $3142567$ has six valid hook configurations, which are shown in Figure \ref{Fig4}. The colorings induced by these valid hook configurations are portrayed in Figure \ref{Fig5}. The valid compositions induced by these valid hook configurations are (reading the first row before the second row, each from left to right) \[(3,1,1),\quad (2,2,1),\quad(1,3,1),\quad(2,1,2),\quad(1,2,2),\quad(1,1,3).\] It follows from Theorem \ref{Thm5} that \[|s^{-1}(3142567)|=C_{(3,1,1)}+C_{(2,2,1)}+C_{(1,3,1)}+C_{(2,1,2)}+C_{(1,2,2)}+C_{(1,1,3)}=27.\] 
\end{example}

It is also possible to refine Theorem \ref{Thm5} according to certain permutation statistics such as the number of descents and the number of peaks\footnote{Theorem \ref{Thm7} was originally stated in \cite{Defant} in terms of ``valleys" instead of peaks, but the formulation we give here is equivalent.}. Recall that a peak of a permutation $\pi=\pi_1\cdots\pi_n\in S_n$ is an index $i$ such that $\pi_{i-1}<\pi_i>\pi_{i+1}$. In what follows, we consider the \emph{Narayana numbers} $N(i,j)=\frac 1i{i\choose j}{i\choose j-1}$. Let us also define \[V(i,j)=2^{i-2j+1}{i-1\choose 2j-2}C_{j-1}.\] It is known\footnote{See sequence A091894	in the Online Encyclopedia of Integer Sequences \cite{OEIS}.} that $V(i,j)$ is the number of decreasing binary plane trees with $i$ vertices and $j$ leaves, and this is actually why these numbers arise in this context. 

\begin{theorem}[\!\!\cite{Defant}]\label{Thm7}
If $\pi\in S_n$ has exactly $k$ descents, then the number of permutations in $s^{-1}(\pi)$ with exactly $m$ descents is \[\sum_{(q_0,\ldots,q_k)\in\mathcal V(\pi)}\sum_{j_0+\cdots+j_k=m+1}\prod_{t=0}^kN(q_t,j_t),\] where the second sum ranges over positive integers $j_0,\ldots,j_k$ that sum to $m+1$. 
The number of permutations in $s^{-1}(\pi)$ with exactly $m$ peaks is
\[\sum_{(q_0,\ldots,q_k)\in\mathcal V(\pi)}\sum_{j_0+\cdots+j_k=m+1}\prod_{t=0}^k V(q_t,j_t),\] where the second sum ranges over positive integers $j_0,\ldots,j_k$ that sum to $m+1$.   
\end{theorem}    

We will often make implicit use of the following result, which is Lemma 3.1 in \cite{Defant2}. 

\begin{theorem}[\!\!\cite{Defant2}]\label{Thm6}
Each valid composition of a permutation $\pi\in S_n$ is induced by a unique valid hook configuration of $\pi$. 
\end{theorem}

In her study of sorted permutations (permutations with positive fertilities), Bousquet-M\'elou introduced the notion of the \emph{canonical tree} of a permutation and showed that the shape of a permutation's canonical tree determines that permutation's fertility \cite{Bousquet}. She asked for an explicit method for computing the fertility of a permutation from its canonical tree. The current author reformulated the notion of a canonical tree in the language of valid hook configurations, defining the \emph{canonical valid hook configuration} of a permutation \cite{Defant2}. Here, we describe a method for computing a permutation's fertility from its canonical valid hook configuration. Specifically, we show how to describe all valid compositions of $\pi$ from the canonical valid hook configuration. This method was stated in \cite{Defant}, but the proof was omitted. We include the proof here because we will make use of this method. In addition, one wishing to program a computer to compute fertilities will likely find that this method is easier to implement than a method that involves finding all valid hook configurations of a permutation directly.  

As before, let $d_1<\cdots<d_k$ be the descents of $\pi$. We will construct the canonical valid hook configuration of $\pi$, which we denote by $\mathcal H^*=(H_1^*,\ldots,H_k^*)$. That is, $H_i^*$ is the hook in $\mathcal H^*$ whose southwest endpoint is $(d_i,\pi_{d_i})$. In order to define $\mathcal H^*$, we need to choose the northeast endpoints of the hooks $H_1^*,\ldots,H_k^*$. To start, consider all possible points that could be northeast endpoints of $H_k^*$; these are precisely the points above and to the right of $(d_k,\pi_{d_k})$. Among these points, choose the leftmost one (equivalently, the lowest one) to be the northeast endpoint of $H_k^*$. Next, consider all possible points that could be northeast endpoints of $H_{k-1}^*$ (given that we already know $H_k^*$ and that we need to satisfy the conditions in Definition \ref{Def2}). Among these points, choose the leftmost one to be the northeast endpoint of $H_{k-1}^*$. Continue in this fashion, always choosing the leftmost possible point as the northeast endpoint of $H_\ell^*$ given that $H_{\ell+1}^*,\ldots,H_k^*$ have already been chosen. If it is ever impossible to find a northeast endpoint for $H_\ell^*$, then $\pi$ has no valid hook configurations (meaning $\pi$ is not sorted). Otherwise, we obtain the canonical valid hook configuration of $\pi$ from this process. Figure \ref{Fig7} shows the canonical valid hook configuration of a permutation. 

\begin{figure}[t]
\begin{center} 
\includegraphics[height=6.0cm]{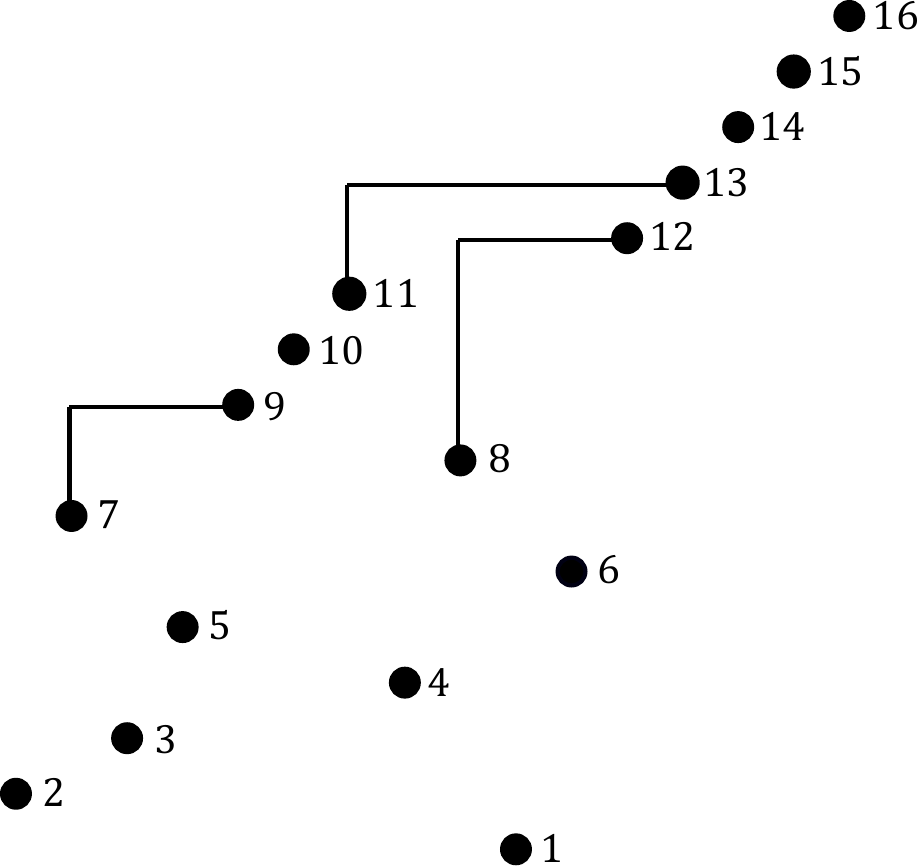}
\end{center}
\captionof{figure}{The canonical valid hook configuration of \\$2\,\,7\,\,3\,\,5\,\,9\,\,10\,\,11\,\,4\,\,8\,\,1\,\,6\,\,12\,\,13\,\,14\,\,15\,\,
16$. } \label{Fig7}
\end{figure}

Let us assume $\pi$ is sorted so that it has a canonical valid hook configuration $\mathcal H^*=(H_1^*,\ldots,H_k^*)$. Recall that $d_1<\cdots<d_k$ are the descents of $\pi$. We make the conventions $d_0=0$ and $d_{k+1}=n$. For $1\leq i\leq k+1$, the $i^\text{th}$ ascending run of $\pi$ is the string $\pi_{d_{i-1}+1}\cdots\pi_{d_i}$. We use $\mathcal H^*$ to define certain parameters as follows. 

\begin{itemize}
\item Let $(b_i^*,\pi_{b_i^*})$ be the northeast endpoint of $H_i^*$. 
\item Let $(q_0^*,\ldots,q_k^*)$ be the valid composition of $\pi$ induced from $\mathcal H^*$.
\item For $1\leq i\leq k$, define $e_i$ by requiring that $\pi_{b_i^*}$ is in the $e_i^\text{th}$ ascending run of $\pi$. In other words, $d_{e_i-1}<b_i^*\leq d_{e_i}$. Furthermore, put $e_0=k+1$. 
\item Let $\alpha_j=\vert\{i\in\{1,\ldots,k\}\colon e_i=j\}\vert$ denote the number of northeast endpoints $(b_i^*,\pi_{b_i^*})$ such that $\pi_{b_i^*}$ is in the $j^{\text{th}}$ ascending run of $\pi$. 
\end{itemize}

\begin{example}\label{Exam}
Let $\pi=2\,\,7\,\,3\,\,5\,\,9\,\,10\,\,11\,\,4\,\,8\,\,1\,\,6\,\,12\,\,13\,\,14\,\,15\,\,
16$ be the permutation whose canonical valid hook configuration appears in Figure \ref{Fig7}. We have $d_0=0$, $d_1=2$, $d_2=7$, $d_3=9$, and $d_4=16$. Furthermore,

\begin{itemize}
\item $(b_1^*,b_2^*,b_3^*)=(5,13,12)$;
\item $(q_0^*,q_1^*,q_2^*,q_3^*)=(7,2,2,2)$;
\item $(e_0,e_1,e_2,e_3)=(4,2,4,4)$; 
\item $(\alpha_1,\alpha_2,\alpha_3,\alpha_4)=(0,1,0,2)$. 
\end{itemize}  
\end{example}

We are now in a position to state and prove the theorem that allows one to combine the above pieces of data in order to describe the valid compositions of $\pi$. The reader who is interested only in the enumerative results of Section \ref{Sec:Back} can safely bypass the proof of the following theorem. 

\begin{theorem}\label{Thm4}
Let $\pi\in S_n$ be a sorted permutation, and preserve the notation from above. A composition $(q_0,\ldots,q_k)$ of $n-k$ into $k+1$ parts is a valid composition of $\pi$ if and only if the following two conditions hold:
\begin{enumerate}[(a)]
\item For every $m\in\{0,\ldots,k\}$, \[\sum_{j=m}^{e_m-1}q_j\geq\sum_{j=m}^{e_m-1}q_j^*.\]
\item For all $m,p\in\{0,1,\ldots,k\}$ with $m\leq p\leq e_m-2$, we have \[\sum_{j=m}^pq_j\geq d_{p+1}-d_m-\sum_{j=m+1}^{p+1}\alpha_j.\]
\end{enumerate}
\end{theorem}

\begin{proof}
To ease notation, let $P(i)=(i,\pi_i)$. Suppose $(q_0,\ldots,q_k)\in\mathcal V(\pi)$, and let $\mathcal H=(H_1,\ldots,H_k)$ be the valid hook configuration inducing $(q_0,\ldots,q_k)$. Let $P(b_i)$ be the northeast endpoint of $H_i$. Put $\pi_0=n+1$, $\pi_{n+1}=n+2$, and $b_0=b_0^*=n+1$. It will be convenient to view the ``sky" as a hook $H_0=H_0^*$ with southwest endpoint $(d_0,\pi_{d_0})=(0,n+1)$ and northeast endpoint $(b_0,\pi_{b_0})=(n+1,n+2)$ (this conflicts with our definition of a hook since these points are not in the plot of $\pi$, but we will ignore this technicality).\footnote{The term ``skyhook" refers to, among other things, a recovery system employed by the United States CIA and Air Force in the $20^{\text{th}}$ century. The system allows an airplane to attach to a cable while in flight. The cable is, in turn, attached to a person, who is then dragged behind the plane. One might find it helpful (or not) to picture an airplane at the point $(n+1,n+2)$, a person at the point $(0,n+1)$, and a blue cable connecting the plane and the person.} Put $\mathcal B^*=\{b_0^*,\ldots,b_k^*\}$ and $\mathcal B=\{b_0,\ldots,b_k\}$. 
If we build $\mathcal H$ by choosing the northeast endpoints $P(b_k),\ldots,P(b_1)$ in this order, then every possible choice for $P(b_m)$ was also a choice for $P(b_m^*)$ when we built $\mathcal H^*$. It follows from our choice of $P(b_m^*)$ that $b_m^*\leq b_m$ and $\pi_{b_m^*}\leq \pi_{b_m}$. This implies that $H_m$ lies above $H_m^*$ or is equal to $H_m^*$ for every $m\in\{0,\ldots,k\}$.

Suppose $m,p\in\{0,\ldots,k\}$ and $m\leq p\leq e_m-1$ (recall that $\pi_{b_m^*}$ is in the $e_m^\text{th}$ ascending run of $\pi$). Let $X=\{d_m+1,d_m+2,\ldots,\min\{b_m^*,d_{p+1}\}\}$. Suppose $b_\gamma\in X\cap\mathcal B$, where $\gamma\neq m$. Because $b_m^*\leq b_m$, we must have $d_m<b_\gamma<b_m$. This means that $b_\gamma$ lies below the hook $H_m$, so $H_\gamma$ lies below $H_m$. Deducing that $m+1\leq\gamma$, we find that $d_m<b_\gamma^*$. We also know that $b_\gamma^*\leq b_\gamma$, so $d_m<b_\gamma^*\leq \min\{b_m^*,d_{p+1}\}$ (because $b_\gamma\in X$). This proves the implication $b_\gamma\in X\cap\mathcal B\Longrightarrow b_\gamma^*\in X\cap\mathcal B^*$. The map $X\cap\mathcal B\to X\cap\mathcal B^*$ given by $b_j\mapsto b_j^*$ is an injection, so 
\begin{equation}\label{Eq10}
\vert X\cap\mathcal B\vert\leq\vert X\cap\mathcal B^*\vert.
\end{equation}

Choose $x\in X\setminus\mathcal B$. Recall that $\mathcal H$ induces a coloring of the plot of $\pi$. The point $P(x)$ lies below the hook $H_m$. None of the hooks $H_0,H_1,\ldots,H_{m-1}$ lie below $H_m$, and all of the hooks $H_{p+1},H_{p+2},\ldots,H_k$ appear to the right of $P(x)$. Therefore, if $P(x)$ looks directly upward, it sees one of the hooks $H_m,H_{m+1},\ldots,H_p$. Letting $\mathcal A_{m,p}$ be the set of points that are given the same color as one of the hooks $H_m,H_{m+1},\ldots,H_p$, we see that $x\in \mathcal A_{m,p}$. This shows that $(X\setminus\mathcal B)\subseteq\mathcal A_{m,p}$. Hence, 
\begin{equation}\label{Eq11}
\sum_{j=m}^pq_j=|\mathcal A_{m,p}|\geq\vert X\setminus\mathcal B\vert.
\end{equation}

Suppose $p\leq e_m-2$ (meaning $\min\{b_m^*,d_{p+1}\}=d_{p+1}$). We have $\vert X\cap\mathcal B^*\vert=\displaystyle{\sum_{j=m+1}^{p+1}\alpha_j}$ and $\vert X\vert=d_{p+1}-d_m$, so we can use \eqref{Eq10} to find that
\[\vert X\setminus\mathcal B\vert=|X|-|X\cap\mathcal B|\geq\vert X\vert-\vert X\cap\mathcal B^*\vert=d_{p+1}-d_m-\sum_{j=m+1}^{p+1}\alpha_j.\]
Combining this last inequality with \eqref{Eq11} yields $(b)$. 

Next, suppose $p=e_m-1$ (meaning $\min\{b_m^*,d_{p+1}\}=b_m^*$). In this case, the elements of $X\setminus\mathcal B^*$ are the indices $x$ such that $P(x)$ lies below $H_m^*$ and is not a northeast endpoint of a hook in $\mathcal H^*$. These points are precisely those that are given the same color as one of the hooks $H_m^*,\ldots,H_{e_m-1}^*$ in the coloring induced by $\mathcal H^*$. The number of such points is $\displaystyle \sum_{j=m}^{e_m-1}q_j^*$, so we deduce from \eqref{Eq10} and \eqref{Eq11} that \begin{equation}\label{Eq3}
\sum_{j=m}^{e_m-1}q_j\geq|X\setminus\mathcal B|\geq|X\setminus\mathcal B^*|=\sum_{j=m}^{e_m-1}q_j^*. 
\end{equation}
This proves $(a)$. 

To prove the converse, suppose we are given a composition $(q_0,\ldots, q_k)$ of $n-k$ into $k+1$ parts that satisfies $(a)$ and $(b)$. We wish to construct a valid hook configuration $\mathcal H=(H_1,\ldots,H_k)$ of $\pi$ that induces the composition $(q_0,\ldots, q_k)$. To do so, it suffices to specify the northeast endpoints of the hooks. Letting $P(b_i)$ denote the northeast endpoint of $H_i$ as before, we see that we need only choose the indices $b_i$. We will choose them in the order $b_k,\ldots,b_1$. 

Let $\ell\in\{1,\ldots,k\}$. Suppose that we have already chosen $b_k,\ldots,b_{\ell+1}$ and that we are now ready to choose $b_\ell$. Let $Z_\ell$ be the set of indices $z\in\{d_\ell+1,\ldots,n\}\setminus\{b_{\ell+1},\ldots,b_k\}$ such that $P(z)$ does not lie below any of the hooks $H_{\ell+1},\ldots,H_k$. Let us write $Z_\ell=\{z_\ell(1),\ldots,z_\ell(\theta_\ell)\}$, where $\theta_\ell=\vert Z_\ell\vert$ and $z_\ell(1)<\cdots<z_\ell(\theta_\ell)$. Put $b_\ell=z_\ell(q_\ell+1)$. 

This choice of $b_\ell$ is actually forced upon us. Indeed, we \emph{must} put $b_\ell=z_\ell(w)$ for some $w$. The points $P(z_\ell(1)),\ldots,P(z_\ell(w-1))$ are precisely the points that see the hook $H_\ell$ when they look directly upward. Therefore, if we can show that this construction actually produces a valid hook configuration $\mathcal H$, we will know that $(q_0,\ldots,q_k)$ is the valid composition of $\pi$ induced by $\mathcal H$. Furthermore, we will know that $\mathcal H$ is the unique valid hook configuration of $\pi$ inducing $(q_0,\ldots,q_k)$ (this is essentially the proof of Theorem \ref{Thm6} given in \cite{Defant2}).  

We first need to verify that $\theta_\ell\geq q_\ell+1$ so that $z_\ell(q_\ell+1)$ actually exists. There are $n-d_\ell-(k-\ell)$ indices $z\in\{d_\ell+1,\ldots,n\}\setminus\{b_{\ell+1},\ldots,b_k\}$, and $\displaystyle{\sum_{j=\ell+1}^{k+1}q_j}$ of them are such that $P(z)$ lies below one of the hooks $H_{\ell+1},\ldots,H_k$. Consequently, 
\begin{equation}\label{Eq7}
\theta_\ell=\vert Z_\ell\vert=n-d_\ell-(k-\ell)-\sum_{j=\ell+1}^kq_j.
\end{equation} 
We know that $b_i^*>d_i$ for all $i$. Therefore, among the numbers $b_1^*,\ldots,b_k^*$, only $b_1^*,\ldots,b_{\ell-1}^*$ could possibly lie in the first $\ell$ ascending runs of $\pi$. This shows that $\displaystyle \sum_{j=1}^\ell\alpha_j\leq\ell-1$. 
Combining this with \eqref{Eq7} and the fact that $\displaystyle{\sum_{j=0}^\ell q_j=n-k-\sum_{j=\ell+1}^kq_j}$ gives 
\[\theta_\ell=\ell-d_\ell+\sum_{j=0}^\ell q_j=\ell-1-d_\ell+\sum_{j=0}^{\ell-1} q_j+q_\ell+1\geq\sum_{j=0}^{\ell-1}q_j-\left(d_\ell-\sum_{j=1}^\ell\alpha_j\right)+q_\ell+1.\]
Setting $m=0$ and $p=\ell-1$ in condition $(b)$ yields \[\sum_{j=0}^{\ell-1}q_j\geq d_\ell-\sum_{j=1}^\ell\alpha_j,\] so $\theta_\ell\geq q_\ell+1$. 

Now that we have defined the indices $b_\ell$, we can construct the hooks $H_\ell$. Specifically, $H_\ell$ is the hook with southwest endpoint $P(d_\ell)$ and northeast endpoint $P(b_\ell)$. To check that this is in fact a hook, we must verify that $d_\ell<b_\ell$ and $\pi_{d_\ell}<\pi_{b_\ell}$ for all $\ell$. We constructed $b_\ell$ so that $d_\ell<b_\ell$. We also know that $\pi_{d_\ell}<\pi_{b_\ell^*}$ because $H_\ell^*$ is a hook with southwest endpoint $P(d_\ell)$ and northeast endpoint $P(b_\ell^*)$. Hence, it suffices to show that $\pi_{b_\ell^*}\leq\pi_{b_\ell}$. 

Observe that $b_k^*=d_k+q_k^*+1$ because the points lying below $H_k^*$ are precisely $P(d_k+1),\ldots,$ $P(d_k+q_k)$. Likewise, $b_k=d_k+q_k+1$. Setting $m=k$ in condition $(a)$ yields $q_k\geq q_k^*$ because $e_k=k+1$. This shows that $b_k^*\leq b_k$. This also forces the inequality $\pi_{b_k^*}\leq \pi_{b_k}$ since $\pi_{b_k^*}$ and $\pi_{b_k}$ both lie in the $(k+1)^\text{th}$ ascending run of $\pi$. It follows that $H_k$ lies above $H_k^*$ or is equal to $H_k^*$. Now, choose some $\ell\in\{1,\ldots,k-1\}$, and suppose that $b_m^*\leq b_m$ and $\pi_{b_m^*}\leq\pi_{b_m}$ for all $m\in\{\ell+1,\ldots,k\}$. For each such $m$, this means that $H_m$ lies above or is equal to $H_m^*$. Recall the definition of $Z_\ell$ from above. It is straightforward to check that the entries $\pi_{z_\ell(1)},\ldots,\pi_{z_\ell(\theta_\ell)}$ are left-to-right maxima of the string $\pi_{d_\ell+1}\cdots\pi_n$ (a left-to-right maximum of a string of positive  integers $w_1\cdots w_m$ is an entry $w_j$ such that $w_j>w_i$ for all $i\in\{1,\ldots,j-1\}$). 

We know from our definition of $e_\ell$ that $b_\ell^*\geq d_{e_\ell-1}+1$. We wish to show that $b_\ell\geq b_\ell^*$, which will of course imply that $b_\ell\geq d_{e_\ell-1}+1$. By way of contradiction, let us first assume $b_\ell\leq d_{e_\ell-1}$. We have $d_p+1\leq b_\ell\leq d_{p+1}$ for some $p\in\{\ell,\ell+1,\ldots,e_\ell-2\}$. Construct the hook $H_\ell$ with southwest endpoint $P(d_\ell)$ and northeast endpoint $P(b_\ell)$ (we do not yet know that $\pi_{d_\ell}<\pi_{b_\ell}$, but we can still connect $P(d_\ell)$ and $P(b_\ell)$ with line segments and call the resulting shape a ``hook"). Because $\pi_{b_\ell}=\pi_{z_\ell(q_\ell+1)}$ is a left-to-right maximum of $\pi_{d_\ell+1}\cdots\pi_n$, every point $P(x)$ with $d_\ell+1\leq x\leq d_p$ lies below $H_\ell$. Furthermore, each of the hooks $H_{\ell+1},\ldots,H_p$ must lie (entirely) below $H_\ell$ because $P(b_\ell)$ cannot lie above any of these hooks (one can verify that our construction guarantees that no point in the plot of $\pi$ lies above any of these hooks). It follows from our construction that there are precisely $\displaystyle{\sum_{j=\ell}^pq_j}$ points that lie below the hook $H_\ell$ and are not in $\{P(b_1),\ldots,P(b_k)\}$. Each of the points $P(b_{\ell+1}),\ldots,P(b_p)$ lies below $H_\ell$ because the hooks $H_{\ell+1},\ldots,H_p$ lie below $H_\ell$. This means that the total number of points lying below $H_\ell$ is at least $p-\ell+\displaystyle{\sum_{j=\ell}^pq_j}$. For each such point $P(z)$, we have $d_\ell+1\leq z\leq d_{p+1}$, so 
\begin{equation}\label{Eq9}
p-\ell+\sum_{j=\ell}^pq_j<d_{p+1}-d_\ell
\end{equation}
Note that the inequality here is strict because $b_\ell$ is an element of $\{d_\ell+1,\ldots,d_{p+1}\}$ and $P(b_\ell)$ does not lie below $H_\ell$. 

If $\delta$ is an index such that $d_\ell+1\leq b_\delta^*\leq d_{p+1}$, then $\delta\leq p$. The point $P(b_\delta^*)$ must lie below $H_\ell^*$ because $d_\ell<b_\delta^*\leq d_{p+1}<b_\ell^*$. Hence, $\ell+1\leq \delta\leq p$. This shows that there are at most $p-\ell$ possible choices for $\delta$, so \[\sum_{j=\ell+1}^{p+1}\alpha_j\leq p-\ell.\] It now follows from \eqref{Eq9} that \[\sum_{j=\ell+1}^{p+1}\alpha_j+\sum_{j=\ell}^pq_j<d_{p+1}-d_\ell,\] which we can see is a contradiction by setting $m=\ell$ in condition $(b)$. We conclude that $b_\ell\geq d_{e_\ell-1}+1$. 

Note that $H_\ell$ lies above the hooks $H_{\ell+1},\ldots,H_{e_\ell-1}$, so the number of points not in the set $\{P(b_1),\ldots,P(b_k)\}$ that lie below $H_\ell$ is at least $\displaystyle{\sum_{j=\ell}^{e_\ell-1}q_j}$. By condition $(a)$, this is at least $\displaystyle{\sum_{j=\ell}^{e_\ell-1}q_j^*}$. Moreover, each of the points $P(b_{\ell+1}),\ldots,P(b_{e_\ell-1})$ lies below $H_\ell$. This shows that there are at least \[e_\ell-1-\ell+\sum_{j=\ell}^{e_\ell-1}q_j^*\] points below $H_\ell$.

If $\eta$ is an index such that $P(b_\eta^*)$ lies below $H_\ell^*$, then $d_\ell<d_\eta<b_\eta^*<b_\ell^*\leq d_{e_\ell}$. This guarantees that $\ell<\eta\leq e_\ell-1$, so there are at most $e_\ell-1-\ell$ such indices $\eta$. The number of points below $H_\ell^*$ that are not of the form $P(b_\eta^*)$ is $\displaystyle{\sum_{j=\ell}^{e_\ell-1}q_j^*}$, so the total number of points below $H_\ell^*$ is at most \[e_\ell-1-\ell+\sum_{j=\ell}^{e_\ell-1}q_j^*.\] According to the previous paragraph, the number of points below $H_\ell$ is at least the number of points below $H_\ell^*$. Therefore, $H_\ell$ lies above $H_\ell^*$ or is equal to $H_\ell^*$. In other words, $b_\ell^*\leq b_\ell$ and $\pi_{b_\ell^*}\leq \pi_{b_\ell}$. It follows by induction that $b_i^*\leq b_i$ and $\pi_{b_i^*}\leq\pi_{b_i}$ for all $i\in\{1,\ldots,k\}$. 

We have shown that the hooks $H_1,\ldots,H_k$ are in fact bona fide hooks. It is straightforward to check that our construction guarantees that $\mathcal H=(H_1,\ldots,H_k)$ is a valid hook configuration of $\pi$, so the proof is complete. 
\end{proof}

\section{Background on Preimages of Permutation Classes}\label{Sec:Back}

In this section, we review some known results concerning preimages of permutation classes. We also establish some notation for subsequent sections.  

First, suppose $\pi\in S_n$ for some $n\geq 4$, and write $\pi=LnR$ so that $s(\pi)=s(L)s(R)n$. Either $L$ or $R$ has length at least $2$, and the permutations $s(L)$ and $s(R)$ each must end in their last entries. It follows that $s(\pi)$ contains the pattern $123$, so \[|s^{-1}(\Av_n(123))|=0\quad\text{whenever}\quad n\geq 4.\] It is easy to see that \[|s^{-1}(\Av_n(213))|=C_n.\] Indeed, suppose $\pi\in s^{-1}(\Av_n(213))$. Since $s(\pi)$ avoids $213$ and has last entry $n$, $s(\pi)$ must be the identity permutation of length $n$. In other words, $s^{-1}(\Av(213))=s^{-1}(\Av(21))=\Av(231)$ by Theorem \ref{Thm1}. Theorem \ref{Thm2} tells us that \[|s^{-1}(\Av_n(231))|=\frac{2}{(n+1)(2n+1)}{3n\choose n}.\] Theorem 3.2 in \cite{Bouvel} states that we also have \[|s^{-1}(\Av_n(132))|=\frac{2}{(n+1)(2n+1)}{3n\choose n}.\] Part of Theorem 3.4 in \cite{Bouvel} states that \[|s^{-1}(\Av_n(312))|=\frac{2}{n(n+1)^2}\sum_{k=1}^n{n+1\choose k-1}{n+1\choose k}{n+1\choose k+1}.\] This last expression is also the number of so-called \emph{Baxter permutations} of length $n$, and it produces the sequence A001181 in the Online Encyclopedia of Integer Sequences \cite{OEIS}. The only length-$3$ pattern $\tau$ for which $|s^{-1}(\Av_n(\tau))|$ is not known is $321$. The sequence $(|s^{-1}(\Av_n(321))|)_{n\geq 1}$ appears to be new. In Section \ref{Sec:321}, we use valid hook configurations to derive estimates for the exponential growth rate of this sequence.   
 
The above remarks show that the sets $s^{-1}(\Av(\tau^{(1)},\ldots,\tau^{(r)}))$ are relatively uninteresting when one of the patterns $\tau^{(i)}$ is an element of $\{123,213\}$. Hence, we will focus our attention on permutation classes whose bases are subsets of $\{132,231,312,321\}$. 

As mentioned in the introduction, it is always possible to describe $s^{-1}(\Av(\tau^{(1)},\ldots,\tau^{(r)}))$ as the set of permutations avoiding some finite collection of \emph{mesh patterns}. We find it simpler to describe our stack-sorting preimage sets in terms of barred patterns or vincular patterns. A \emph{barred pattern} is a permutation pattern in which some entries are overlined. Saying a permutation contains a barred pattern means that it contains a copy of the pattern formed by the unbarred entries that is not part of a pattern that has the same relative order as the full barred pattern. For example, saying a permutation contains the barred pattern $3\overline{5}241$ means that it contains a $3241$ pattern that is not part of a $35241$ pattern. In fact, West \cite{West} introduced barred patterns in order to describe $2$-stack-sortable permutations, showing that a permutation is $2$-stack-sortable if and only if it avoids the classical pattern $2341$ and the barred pattern $3\overline{5}241$. 

A \emph{vincular pattern} is a permutation pattern in which some consecutive entries can be underlined. We say a permutation \emph{contains} a vincular pattern if it contains an occurrence of the permutation pattern in which underlined entries are consecutive. For example, saying that a permutation $\sigma=\sigma_1\cdots\sigma_n$ contains the vincular pattern $\underline{32}41$ means that there are indices $i_1<i_2<i_3<i_4$ such that $\sigma_{i_4}<\sigma_{i_2}<\sigma_{i_1}<\sigma_{i_3}$ and $i_2=i_1+1$. Vincular patterns appeared first in \cite{Babson} and have received a large amount of attention ever since \cite{Steingrimsson}. 

If $\tau$ is a classical, barred, or vincular pattern, then we say a permutation \emph{avoids} $\tau$ if it does not contain $\tau$. Let $\Av(\tau^{(1)},\ldots,\tau^{(r)})$ be the set of permutations avoiding $\tau^{(1)},\ldots,\tau^{(r)}$, and let $\Av_n(\tau^{(1)},\ldots,\tau^{(r)})=\Av(\tau^{(1)},\ldots,\tau^{(r)})\cap S_n$.


Recall the notation from Theorem \ref{Thm5} and Theorem \ref{Thm7}. When counting preimages of permutations according to numbers of descents and peaks, we will make use of the generating functions \[F(x,y)=\sum_{n\geq 1}\sum_{m\geq 1}N(n,m)x^ny^{m-1}\quad\text{and}\quad G(x,y)=\sum_{n\geq 1}\sum_{m\geq 1}V(n,m)x^ny^{m-1}.\] It is known that 
\begin{equation}\label{Eq2}
F(x,y)=\frac{1-x(y+1)-\sqrt{1-2x(y+1)+x^2(y-1)^2}}{2xy}
\end{equation}
and 
\begin{equation}\label{Eq4}
G(x,y)=\frac{1-2x-\sqrt{(1-2x)^2-4x^2y}}{2xy}.
\end{equation}
We let $[z_1^{n_1}\cdots z_r^{n_r}]A(z_1,\ldots,z_r)$ denote the coefficient of $z_1^{n_1}\cdots z_r^{n_r}$ in the generating function \linebreak $A(z_1,\ldots,z_r)$. 

\section{$s^{-1}(\Av(132,231,312,321))$}\label{Sec:4patterns}

It turns out that $s^{-1}(\Av(132,231,312,321))$ is an actual permutation class; we have \[s^{-1}(\Av(132,231,312,321))=\Av(1342,2341,3142,3241,3412,3421).\] One could probably enumerate this class directly, but we will use valid hook configurations in order to illustrate this uniform method for finding fertilities.  

\begin{theorem}\label{Thm3}
For $n\geq 2$, we have \[|s^{-1}(\Av_n(132,231,312,321))|=2C_n-2C_{n-1}.\] The number of elements of $s^{-1}(\Av_n(132,231,312,321))$ with $m$ descents is \[N(n,m+1)+\sum_{i=1}^{n-2}\sum_{j=1}^mN(n-i-1,j)N(i,m-j+1).\] The number of elements of $s^{-1}(\Av_n(132,231,312,321))$ with $m$ peaks is \[V(n,m+1)+\sum_{i=1}^{n-2}\sum_{j=1}^mV(n-i-1,j)V(i,m-j+1).\]
\end{theorem}

\begin{proof}
The only elements of $\Av_n(132,231,312,321)$ are $123\cdots n$ and $2134\cdots n$. The only valid composition of $123\cdots n$ is $(n)$. Each valid hook configuration of $2134\cdots n$ has exactly one hook. This hook has southwest endpoint $(1,2)$ and has northeast endpoint $(n+1-i,n+1-i)$ for some $i\in\{1,\ldots,n-2\}$. This valid hook configuration induces the valid composition $(i,n-i-1)$. It follows from Theorem \ref{Thm5} and the standard Catalan number recurrence relation that \[|s^{-1}(\Av_n(132,231,312,321))|=C_n+\sum_{i=1}^{n-2}C_iC_{n-i-1}=C_n+\sum_{i=0}^{n-1}C_iC_{n-i-1}-2C_{n-1}=2C_n-2C_{n-1}.\] The second and third statements of the theorem follow immediately from Theorem \ref{Thm7}. 
\end{proof}

\section{$s^{-1}(\Av(132,231,321))$ and $s^{-1}(\Av(132,312,321))$}\label{Sec5} 

It turns out that \[s^{-1}(\Av(132,231,321))=\Av(1342,2341,3241,45231,3\overline{5}142).\] The set $s^{-1}(\Av(132,312,321))$ is actually equal to the permutation class $\Av(1342,3142,3412,3421)$. 

Let \[\sigma_{n,\ell}=\ell 12\cdots(\ell-1)(\ell+1)\cdots n\quad\text{and}\quad\gamma_{n,\ell}=23\cdots(\ell-1)1(\ell+1)\cdots n.\] It is straightforward to check that $\Av_n(132,231,321)=\{\sigma_{n,1},\ldots,\sigma_{n,n}\}$ and $\Av_n(132,312,321)=\{\gamma_{n,1},\ldots,\gamma_{n,n}\}$. For example, $\Av_4(132,231,321)=\{1234,2134,3124,4123\}$. 
West \cite{West} found formulas for the fertilities of $\sigma_{n,\ell}$ and $\gamma_{n,\ell}$ and found that they are equal. It follows that \[|s^{-1}(\Av_n(132,231,321))|=|s^{-1}(\Av_n(132,312,321))|.\] 
This equality is easy to verify with the theory of valid hook configurations. Indeed, for $2\leq \ell\leq n$, we have \[\mathcal V(\sigma_{n,\ell})=\{(n-\ell-i+1,\ell+i-2):1\leq i\leq n-\ell\}\hspace{.29cm}\text{and}\hspace{.29cm}\mathcal V(\gamma_{n,\ell})=\{(\ell+i-2,n-\ell-i+1):1\leq i\leq n-\ell\}.\] 
That is, the valid compositions of $\sigma_{n,\ell}$ are obtained by interchanging the two parts in the valid compositions of $\gamma_{n,\ell}$. Along with Theorem \ref{Thm7}, this also implies the following. 
\begin{theorem}\label{Thm15}
The number of elements of $s^{-1}(\Av_n(132,231,321))$ with $m$ descents is equal to the number of elements of $s^{-1}(\Av_n(132,312,321))$ with $m$ descents. The number of elements of $s^{-1}(\Av_n(132,231,321))$ with $m$ peaks is equal to the number of elements of $s^{-1}(\Av_n(132,312,321))$ with $m$ peaks.
\end{theorem} 
Note that West did not prove these refined equalities. 

\begin{theorem}\label{Thm16}
We have \[|s^{-1}(\Av_n(132,231,321))|=|s^{-1}(\Av_n(132,312,321))|={2n-2\choose n-1}.\] The number of elements of $s^{-1}(\Av_n(132,231,321))$ (equivalently, $s^{-1}(\Av_n(132,312,321))$) with $m$ descents is \[{n-1\choose m}^2.\] The number of elements of $s^{-1}(\Av_n(132,231,321))$ (equivalently, $s^{-1}(\Av_n(132,312,321))$) with $m$ peaks is \[2^{n-2m-2}{n\choose 2m+2}{2m+2\choose m+1}.\]
\end{theorem}

\begin{proof}
By Theorem \ref{Thm15}, we need only consider the preimage sets $s^{-1}(\Av_n(132,231,321))$. We will prove the second and third statements; the first statement will then follow from the second and the well-known identity $\sum_{m=0}^{n-1}{n-1\choose m}^2={2n-2\choose n-1}$. The only valid composition of $\sigma_{n,1}=123\cdots n$ is $(n)$. For $2\leq \ell\leq n$, the valid compositions of $\sigma_{n,\ell}$ are $(n-\ell-i+1,\ell+i-2)$ for $1\leq i\leq n-\ell$. In particular, we can ignore $\sigma_{n,n}$ because it has no valid compositions (that is, $\sigma_{n,n}$ is not sorted).  

Using the first part of Theorem \ref{Thm7}, we find that the number of elements of 
\linebreak $s^{-1}(\Av_n(132,231,321))$ with $m$ descents is \[N(n,m+1)+\sum_{\ell=2}^{n-1}\sum_{(q_0,q_1)\in\mathcal V(\sigma_\ell)}\sum_{j_0+j_1=m+1}N(q_0,j_0)N(q_1,j_1)\] \[=N(n,m+1)+\sum_{\ell=2}^{n-1}\sum_{i=1}^{n-\ell}\sum_{j=1}^m N(n-\ell-i+1,j)N(\ell+i-2,m-j+1).\] Letting $r=\ell+i$, this becomes 
\[N(n,m+1)+\sum_{\ell=2}^{n-1}\sum_{r=\ell+1}^n\sum_{j=1}^mN(n-r+1,j)N(r-2,m-j+1)\] \[=N(n,m+1)+\sum_{\ell=1}^{n-2}\sum_{r=\ell}^{n-2}\sum_{j=1}^m N(n-r-1,j)N(r,m-j+1)\] \[=N(n,m+1)+\sum_{r=1}^{n-2}r\sum_{j=1}^m N(n-r-1,j)N(r,m-j+1)\] \[=[x^ny^m]\left(F(x,y)+x^2y\,F(x,y)\cdot\frac{\partial}{\partial x}F(x,y)\right).\] Applying \eqref{Eq2} and some algebraic manipulations, we find that this is equal to \[[x^ny^m]\left(\frac{x}{\sqrt{1+x^2(y-1)^2-2x(y+1)}}\right).\] Standard methods now allow us to see that this expression is equal to ${n-1\choose m}^2$.  

The proof of the third statement in the theorem proceeds exactly as in the proof of the second statement. In this case, we find that the number of elements of $s^{-1}(\Av_n(132,231,321))$ with $m$ peaks is \[[x^ny^m]\left(G(x,y)+x^2y\,G(x,y)\cdot\frac{\partial}{\partial x}G(x,y)\right).\]
Applying \eqref{Eq4} and some algebraic manipulations, we find that this is equal to \[[x^ny^m]\left(\frac{x}{\sqrt{1-4x-4x^2(y-1)}}\right).\] Standard methods now allow us to see that this expression is equal to $2^{n-2m-2}{n\choose 2m+2}{2m+2\choose m+1}$.  
\end{proof}

Recently, Bruner proved that \[|\Av_n(2431, 4231, 1432, 4132)|={2n-2\choose n-1}.\] She also listed several other permutation classes that appear to be enumerated by central binomial coefficients but did not prove that this is the case. One of these classes is $\Av(1243,2143,2413,2431)$. Of course, a permutation is in this class if and only if its reverse is in the class \[\Av(1342,3142,3412,3421)=s^{-1}(\Av(132,312,321)).\] Therefore, the following corollary, which follows immediately from Theorem \ref{Thm16} and the discussion immediately preceding that theorem, settles one of the enumerative problems that Bruner listed.

\begin{corollary}
We have \[|\Av_n(1342,3142,3412,3421)|={2n-2\choose n-1}.\] The number of elements of $\Av_n(1342,3142,3412,3421)$ with $m$ descents is \[{n-1\choose m}^2.\] The number of elements of $\Av_n(1342,3142,3412,3421)$ with $m$ peaks is \[2^{n-2m-2}{n\choose 2m+2}{2m+2\choose m+1}.\]
\end{corollary}

\section{$s^{-1}(\Av(231,312,321))$}\label{Sec6} 

We will find it convenient to identify a permutation with a configuration of points in the plane via its plot. Doing so, we can build permutations by placing the plots of smaller permutations in various configurations. For example, if $\lambda=\lambda_1\cdots\lambda_\ell\in S_\ell$ and $\mu=\mu_1\ldots\mu_m\in S_m$, then the \emph{sum} of $\lambda$ and $\mu$, denoted $\lambda\oplus\mu$, is obtained by placing the plot of $\mu$ above and to the right of the plot of $\lambda$. More formally, the $i^\text{th}$ entry of $\lambda\oplus\mu$ is \[(\lambda\oplus\mu)_i=\begin{cases} \lambda_i & \mbox{if } 1\leq i\leq \ell; \\ \mu_{i-\ell}+\ell & \mbox{if } \ell+1\leq i\leq \ell+m. \end{cases}\] 

Let $\text{Dec}_a=a(a-1)\cdots 1\in S_a$ denote the decreasing permutation of length $a$. The permutations in $\Av(231,312)$ are called \emph{layered}; each is of the form $\text{Dec}_{a_1}\oplus\cdots\oplus \text{Dec}_{a_t}$ for some composition $(a_1,\ldots,a_t)$. For example, $32154687=321\oplus 21\oplus 1\oplus 21$ is the layered permutation corresponding to the composition $(3,2,1,2)$. Under this correspondence between layered permutations and compositions, the permutations in $\Av(231,312,321)$ correspond to compositions whose parts are all at most $2$. It follows that these permutations are counted by the Fibonacci numbers.  

The set $s^{-1}(\Av(231,312,321))$ is actually just the permutation class $\Av(2341,3241,3412,3421)$. A permutation $\sigma$ is in this class if and only if the first and third entries in every $231$ pattern in $\sigma$ are consecutive integers. For example, $236541$ is not in this class because the entries $3,6,1$ form a $231$ pattern while $3$ and $1$ are not consecutive integers. In this section, we use valid hook configurations to derive a formula for $|s^{-1}(\Av_n(231,312,321))|$. We then enumerate this class directly; showing that these permutations are counted by the terms in sequence A049124 in the Online Encyclopedia of Integer Sequences \cite{OEIS}. Together, these results give a new formula and a new combinatorial interpretation for the terms in this sequence. 

Let $\Comp_a(b)$ denote the set of all compositions of $b$ into $a$ parts (that is, $a$-tuples of positive integers that sum to $b$). Define a partial order $\preceq$ on $\Comp_a(b)$ by declaring that $(x_1,\ldots,x_a)\preceq(y_1,\ldots,y_a)$ if $\sum_{i=1}^\ell x_i\leq\sum_{i=1}^\ell y_i$ for all $\ell\in\{1,\ldots,a\}$. A \emph{partition} is a composition whose parts are nonincreasing. Following \cite{Stanley}, we let $L(u,v)$ denote the set of all partitions (including the empty partition) with at most $u$ parts and with largest part at most $v$. Endow $L(u,v)$ with a partial order $\leq$ by declaring that $(\lambda_1,\ldots,\lambda_\ell)\leq(\mu_1,\ldots,\mu_m)$ if $\ell\leq m$ and $\lambda_i\leq\mu_i$ for all $i\in\{1,\ldots,\ell\}$. Geometrically, $L(u,v)$ is the set of all partitions whose Young diagrams fit inside a $u\times v$ rectangle, and $\lambda\leq\mu$ if and only if the Young diagram of $\lambda$ fits inside of the Young diagram of $\mu$. 

Given $x=(x_1,\ldots,x_a)\in\Comp_a(b)$, let $\psi(x)\in L(b-a,a-1)$ be the partition that has exactly $x_i-1$ parts of length $a-i$ for all $i\in\{1,\ldots,a-1\}$. The map $\psi:\Comp_a(b)\to L(b-a,a-1)$ is an isomorphism of posets. For $x\in\Comp_a(b)$, let \[D_x=|\{y\in\Comp_a(b):y\preceq x\}|.\] Equivalently, $D_{\psi^{-1}(\lambda)}$ is the number of partitions (including the empty partition) whose Young diagrams fit inside of the Young diagram of the partition $\lambda$. Recall the notation $C_{(x_0,\ldots,x_k)}=\prod_{t=0}^k C_{x_t}$, where $C_j$ is the $j^\text{th}$ Catalan number. 
 
\begin{theorem}\label{Thm8}
Preserving the notation of the preceding paragraph with $a=k+1$ and $b=n-k$, we have \[|s^{-1}(\Av_n(231,312,321))|=\sum_{k=0}^{n-1}\sum_{q\,\in\,\Comp_{k+1}(n-k)}C_qD_q=\sum_{k=0}^{n-1}\sum_{\lambda\in L(n-2k-1,k)}C_{\psi^{-1}(\lambda)}D_{\psi^{-1}(\lambda)}.\]
\end{theorem}

\begin{proof}
Recall that each permutation in $\Av(231,312,321)$ is of the form $\text{Dec}_{a_1}\oplus\cdots\oplus\text{Dec}_{a_t}$, where $a_i\in\{1,2\}$ for all $i$. Suppose $\pi\in\Av_n(231,312,321)$ has descents $d_1<\cdots<d_k$. For $p\in\{0,\ldots,k-1\}$, let $u_p=d_{p+1}-p$. Defining $y_0=u_0$, $y_i=u_i-u_{i-1}$ for $1\leq i\leq k-1$, and $y_k=n-k-u_{k-1}$, we obtain a composition $y=(y_0,\ldots,y_k)\in\Comp_{k+1}(n-k)$. Given this composition $y$, we can easily reconstruct the permutation $\pi$. Thus, there is a bijective correspondence between compositions in $\Comp_{k+1}(n-k)$ and permutations in $\Av_n(231,312,321)$ with $k$ descents. 

We are going to use Theorem \ref{Thm4} to describe all of the valid compositions of $\pi$. Preserve the notation from that theorem and the discussion immediately preceding it. We must compute the canonical valid hook configuration $\mathcal H^*=(H_1^*,\ldots,H_k^*)$ of $\pi$. This is fairly simple to do: the hook $H_i^*$ has southwest endpoint $(d_i,\pi_{d_i})$ and northeast endpoint $(d_i+2,\pi_{d_i+2})$. Thus, $b_i^*=d_i+2$. We also have $(q_0^*,\ldots,q_k^*)=(n-2k,1,\ldots,1)$, $e_0=k+1$, and $e_i=i+1$ for all $i\in\{1,\ldots,k\}$. Finally, $\alpha_1=0$, and $\alpha_i=1$ for all $i\in\{2,\ldots,k+1\}$.

Every composition $(q_0,\ldots,q_k)\in\Comp_{k+1}(n-k)$ satisfies condition $(a)$ in Theorem \ref{Thm4}. In condition $(b)$, the inequality $m\leq p\leq e_m-2$ is only satisfied when $m=0$. When $m=0$, \[d_{p+1}-d_m-\sum_{j=m+1}^{p+1}\alpha_j=d_{p+1}-p=u_p.\] Hence, Theorem \ref{Thm4} tells us that a composition $q=(q_0,\ldots,q_k)\in\Comp_{k+1}(n-k)$ is a valid composition of $\pi$ if and only if \[\sum_{j=0}^pq_j\geq u_p\] for all $p\in\{0,\ldots,k-1\}$. This occurs if and only if $y\preceq q$. 

Combining these observations with Theorem \ref{Thm5}, we find that (recall the definition of 
\linebreak $\Av_{n,k}(\tau^{(1)},\ldots,\tau^{(r)})$ from Definition \ref{Def1}) \[|s^{-1}(\Av_n(231,312,321))|=\sum_{k=0}^{n-1}\,\sum_{\pi\in\Av_{n,k}(231,312,321)}\,\sum_{q\in\mathcal V(\pi)}C_q=\sum_{k=0}^{n-1}\sum_{y\in\Comp_{k+1}(n-k)}\sum_{\substack{q\in\Comp_{k+1}(n-k)\\y\preceq q}}C_q\]
\[=\sum_{k=0}^{n-1}\sum_{q\,\in\,\Comp_{k+1}(n-k)}C_qD_q.\] The identity \[\sum_{k=0}^{n-1}\sum_{q\,\in\,\Comp_{k+1}(n-k)}C_qD_q=\sum_{k=0}^{n-1}\sum_{\lambda\in L(n-2k-1,k)}C_{\psi^{-1}(\lambda)}D_{\psi^{-1}(\lambda)}\] follows from the discussion immediately preceding this theorem (with $a=k+1$ and $b=n-k$).        
\end{proof}

\begin{theorem}\label{Thm9}
For $n\geq 1$, we have \[|s^{-1}(\Av_n(231,312,321))|=\sum_{k=0}^{n-1}\frac{1}{n+1}{n-k-1\choose k}{2n-2k\choose n}.\]
\end{theorem}

\begin{proof}
Let $h(n)=|s^{-1}(\Av_n(231,312,321))|=|\Av_n(2341,3241,3412,3421)|$ (with $h(0)=1$), and put \[H(x)=\sum_{n\geq 0}h(n)x^n.\] We will prove that 
\begin{equation}\label{Eq6}
H(x)=1+\frac{xH(x)^2}{1-x^2H(x)^2}.
\end{equation}
Lagrange inversion (the details of which we omit) then allow one to extract the desired formula from this functional equation for $H(x)$. In fact, \[\sum_{k=0}^{n-1}\frac{1}{n+1}{n-k-1\choose k}{2n-2k\choose n}\] is the $(n+1)^\text{th}$ term of the sequence A049124 in the Online Encyclopedia of Integers Sequences \cite{OEIS}, and the generating function $A(x)$ of that sequence satisfies $A(x)=x+\dfrac{A(x)^2}{1-A(x)^2}$ (implying that $A(x)=xH(x)$). 

We can rewrite \eqref{Eq6} as 
\begin{equation}\label{Eq5}
H(x)=1+xH(x)^2+x^2H(x)^2(H(x)-1);
\end{equation}
it is this form of the equation that we will prove. For convenience, say a permutation is \emph{special} if it is an element of $s^{-1}(\Av(231,312,321))$ (including the empty permutation). As mentioned above, a permutation $\sigma=\sigma_1\cdots\sigma_n$ is special if and only if the first and third entries in every $231$ pattern in $\sigma$ are consecutive integers. Note that the empty permutation is special and accounts for the term $1$ on the right-hand side of \eqref{Eq5}. We now show how to construct nonempty special permutations. It is convenient to split these permutations into two types. 

The first type of nonempty special permutation $\sigma$ is that in which the first entry $\sigma_1$ is not part of a $231$ pattern. As mentioned above, we identify permutations with their plots so that we can build large permutations by arranging plots of smaller permutations. Each nonempty special permutation of the first type is formed by choosing two special permutations $\eta$ and $\zeta$ and arranging them as follows: 
\begin{center}
\includegraphics[width=.15\linewidth]{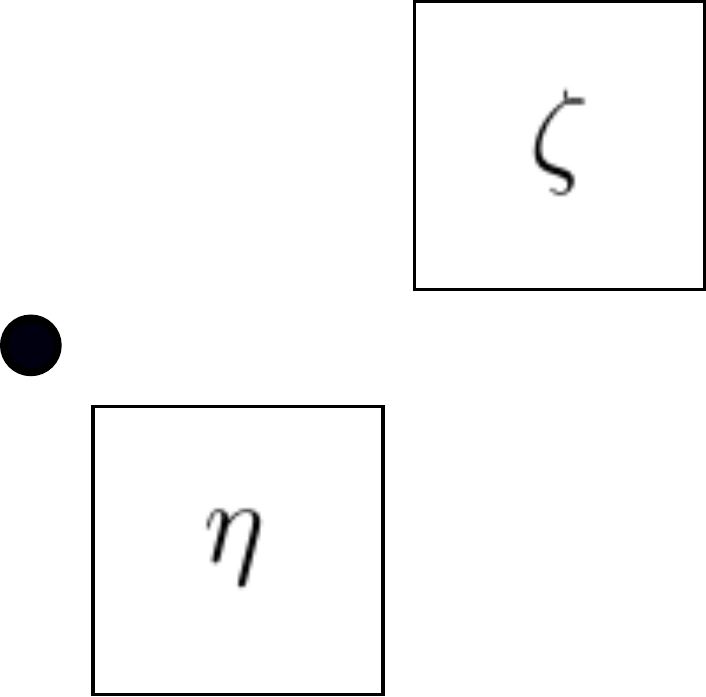}.
\end{center} 
Accordingly, the nonempty special permutations of the first type contribute the term $xH(x)^2$ to the right-hand side of \eqref{Eq5}. 
 
The second type of nonempty special permutation $\sigma$ is that in which $\sigma_1$ is part of a $231$ pattern. We describe how to build such a permutation, leaving the reader to check that every nonempty special permutation of the second type is built uniquely via this procedure. 

Begin by choosing special permutations $\lambda=\lambda_1\cdots\lambda_\ell$, $\tau=\tau_1\cdots\tau_t$, and $\mu=\mu_1\cdots\mu_m$ such that $\mu$ is nonempty. Write $\mu'=\mu_2\cdots\mu_m$ (so $\mu'$ is empty if $m=1$). Write $\lambda=\lambda'\lambda''$, where $\lambda'$ is the first descending run of $\lambda$. In other words, $\lambda'=\lambda_1\cdots\lambda_j$, where $j$ is the smallest index that is not a descent of $\lambda$. If $\lambda$ is empty, then $\lambda'$ and $\lambda''$ are also empty. The plot of $\sigma$ is formed by arranging the plots of $\lambda,\tau,$ and $\mu$ as follows: 
\begin{center}
\includegraphics[width=.35\linewidth]{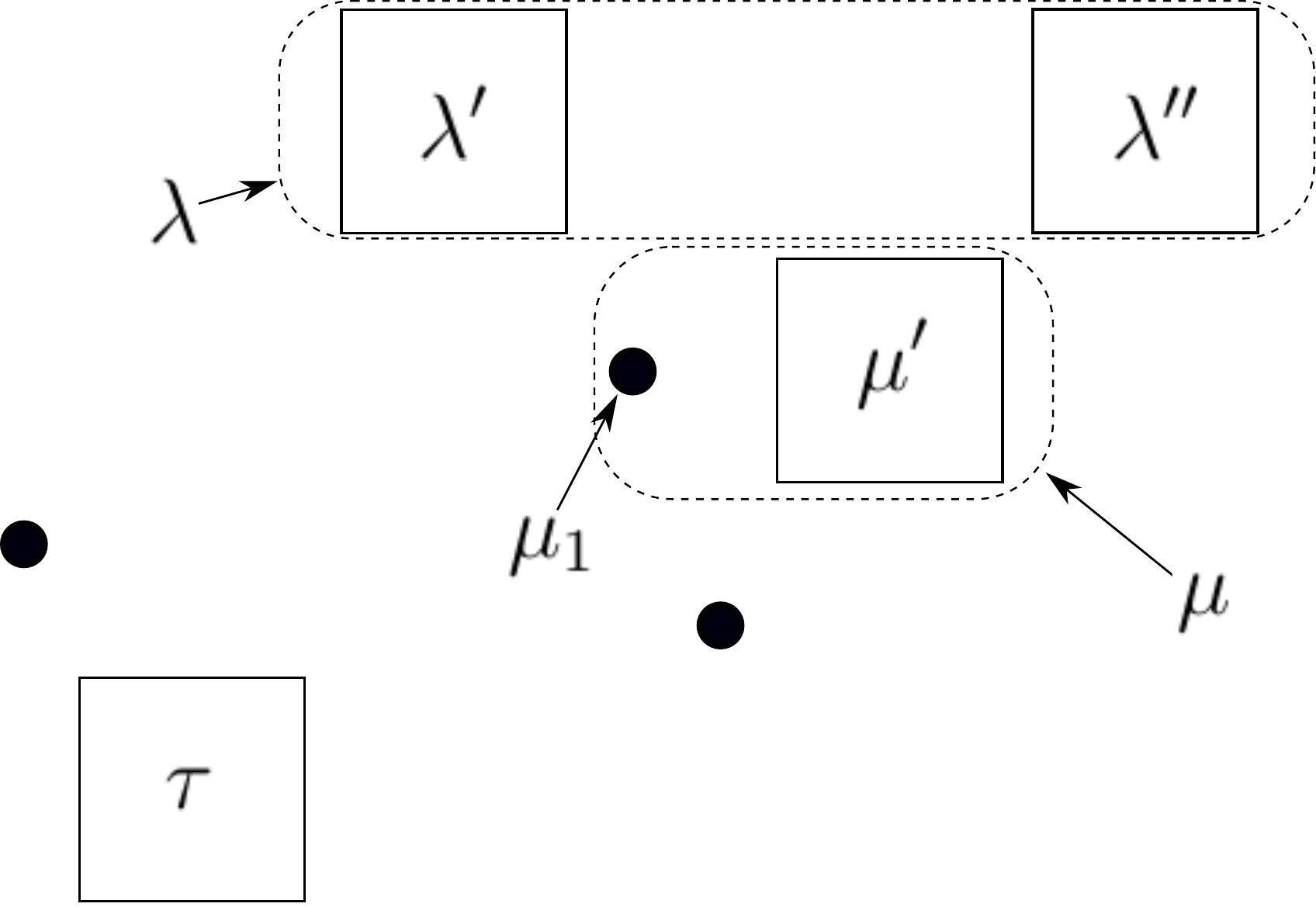}.
\end{center} 
The points in the section labeled $\mu$ are arranged vertically so that they form a permutation that is order isomorphic to $\mu$. In other words, the point labeled $\mu_1$ should be placed higher than exactly $\mu_1-1$ of the points in the box labeled $\mu'$. Similarly, the points in the section labeled $\lambda$ should form a permutation that is order isomorphic to $\lambda$. 

This construction shows that the nonempty special permutations of the second type contribute the term $x^2H(x)^2(H(x)-1)$ to the right-hand side of \eqref{Eq5}.      
\end{proof}

\begin{example}\label{Exam2}
Let us illustrate the part of the proof of Theorem \ref{Thm9} that describes the construction of a nonempty special permutation of the second type. Choose $\tau=21$, $\lambda=31542$, and $\mu=231$. The resulting permutation is $4\,\,2\,\,1\,\,10\,\,8\,\,6\,\,3\,\,7\,\,5\,\,12\,\,11\,\,9$, whose plot is 
\begin{center}
\includegraphics[width=.29\linewidth]{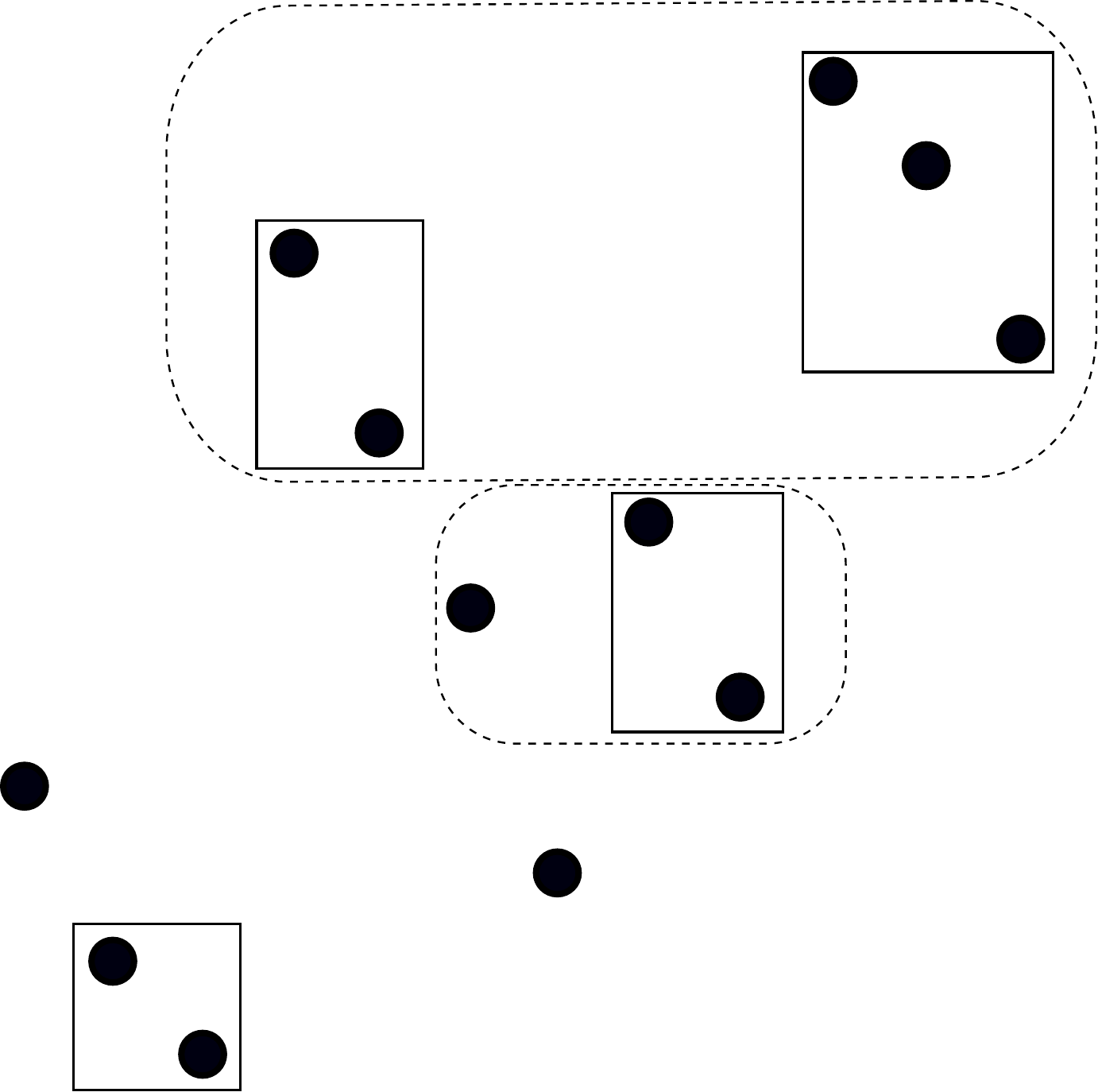}.
\end{center} 
The extra boxes drawn in this plot are intended to highlight the appearances of $\tau,\lambda$, and $\mu$ in the plot. 
\end{example}

Combining Theorem \ref{Thm8} and Theorem \ref{Thm9} yields the identity \[\sum_{k=0}^{n-1}\sum_{q\,\in\,\Comp_{k+1}(n-k)}C_qD_q=\sum_{k=0}^{n-1}\frac{1}{n+1}{n-k-1\choose k}{2n-2k\choose n}.\] Numerical evidence suggests the following conjecture. 

\begin{conjecture}\label{Conj1}
In the notation of Theorem \ref{Thm8}, we have \[\sum_{q\,\in\,\Comp_{k+1}(n-k)}C_qD_q=\frac{1}{n+1}{n-k-1\choose k}{2n-2k\choose n}\] for all nonnegative integers $n$ and $k$. 
\end{conjecture}

When combined with Theorem \ref{Thm8}, a proof of Conjecture \ref{Conj1} would yield an alternative proof of Theorem \ref{Thm9}. We know from the proof of Theorem \ref{Thm8} that \[|s^{-1}(\Av_{n,k}(231,312,321))|=\sum_{q\,\in\,\Comp_{k+1}(n-k)}C_qD_q,\] so Conjecture \ref{Conj1} is equivalent to the identity \[|s^{-1}(\Av_{n,k}(231,312,321))|=\frac{1}{n+1}{n-k-1\choose k}{2n-2k\choose n}.\]

\section{$s^{-1}(\Av(132,231,312))$}\label{Sec7} 

It turns out that \[s^{-1}(\Av(132,231,312))=\Av(2341,3412,1342,3142,34\underline{21},\underline{32}41).\]
In this subsection, we make use of the \emph{generalized Narayana numbers} \[N_k(n,r)=\frac{k+1}{n}{n\choose r+k}{n\choose r-1}.\] Note that the standard Narayana numbers are simply $N(n,r)=N_0(n,r)$.

Finding the fertilities of permutations in $\Av(132,231,312)$ is interesting because it proves the tightness of certain estimates that were used in \cite{Defant2} in order to obtain upper bounds for $W_3(n)$ and $W_4(n)$. More specifically, the following theorem is an immediate consequence of Corollary 3.1 and Lemma 4.1 in \cite{Defant2}. 

\begin{theorem}[\!\!\cite{Defant2}]\label{Thm14}
Suppose $\pi\in S_n$ has $k$ descents. We have \[|s^{-1}(\pi)|\leq\frac{2k+2}{n+1}{2n-2k-1\choose n}.\] Furthermore, the number of elements of $s^{-1}(\pi)$ with exactly $m$ descents is at most \[N_k(n-k,m-k+1)=\frac{k+1}{n-k}{n-k\choose m+1}{n-k\choose m-k}.\] 
\end{theorem}

It turns out that these estimates are sharp when $\pi\in\text{Av}_n(132,231,312)$. In fact, it is straightforward to check that the only permutation in $\text{Av}_n(132,231,312)$ with exactly $k$ descents is \[\theta_{n,k}=(k+1)k(k-1)\cdots 321(k+2)(k+3)\cdots n\] (this permutation is the sum of a decreasing permutation of length $k+1$ and an increasing permutation of length $n-k-1$). For example, $\theta_{7,2}=3214567$ is the only permutation of length $7$ that has $2$ descents and avoids the patterns $132,231,312$. 

\begin{theorem}\label{Thm10}
With notation as above, \[|s^{-1}(\theta_{n,k})|=\frac{2k+2}{n+1}{2n-2k-1\choose n}.\] Furthermore, the number of permutations in $s^{-1}(\theta_{n,k})$ with exactly $m$ descents is \[N_k(n-k,m-k+1).\] 
\end{theorem} 
\begin{proof}
As in Section \ref{Sec6}, let $\Comp_a(b)$ denote the set of compositions of $b$ into $a$ parts. Corollary 3.1 in \cite{Defant} states that \[\sum_{(q_0,\ldots,q_k)\in\Comp_{k+1}(n-k)}C_{(q_0,\ldots,q_k)}=\frac{2k+2}{n+1}{2n-2k-1\choose n}.\] Lemma 4.1 in that same paper tells us that \[\sum_{(q_0,\ldots,q_k)\in\Comp_{k+1}(n-k)}\,\,\sum_{(j_0,\ldots,j_k)\in\Comp_{k+1}(m+1)}\,\,\prod_{t=0}^kN(q_t,j_t)=N_k(n-k,m-k+1).\] Invoking Theorems \ref{Thm5} and \ref{Thm7}, we see that is suffices to show that $\mathcal V(\theta_{n,k})=\Comp_{k+1}(n-k)$. We know that $\mathcal V(\theta_{n,k})\subseteq\Comp_{k+1}(n-k)$, so it remains to prove the reverse containment. This is actually quite simple to do if we apply Theorem \ref{Thm4} with the permutation $\theta_{n,k}$ in place of $\pi$. Preserve the notation from that theorem and the discussion immediately preceding it. 

First, we have $d_i=i$ for $0\leq i\leq k$. It is straightforward to compute the canonical valid hook configuration $\mathcal H^*$ of $\theta_{n,k}$. Namely, $\mathcal H^*=(H_1^*,\ldots,H_k^*)$, where $H_i^*$ is the hook with southwest endpoint $(i,k-i+2)$ and northeast endpoint $(2k-i+2,2k-i+2)$. The valid composition induced by $\mathcal H^*$ is \[(q_0^*,\ldots,q_k^*)=(n-2k,1,\ldots,1).\] We have $e_i=k+1$ for all $i\in\{0,\ldots,k\}$. Finally, $\alpha_i=0$ for $i\in\{1,\ldots,k\}$, and $\alpha_{k+1}=k$. 

Choose $(q_0,\ldots,q_k)\in\Comp_{k+1}(n-k)$. If $m\in\{1,\ldots,k\}$, then \[\sum_{j=m}^{e_m-1}q_j\geq\sum_{j=m}^{e_m-1}1=\sum_{j=m}^{e_m-1}q_j^*.\] If $m=0$, then \[\sum_{j=m}^{e_m-1}q_j=n-k=\sum_{j=m}^{e_m-1}q_j^*.\] Hence, $(q_0,\ldots,q_k)$ satisfies condition $(a)$ in Theorem \ref{Thm4}. If $m,p\in\{0,1,\ldots,k\}$ and $m\leq p\leq e_m-2=k-1$, then \[\sum_{j=m}^p q_j\geq p-m+1=d_{p+1}-d_m-\sum_{j=m+1}^{p+1}\alpha_j.\] This shows that $(q_0,\ldots,q_k)$ satisfies condition $(b)$ in Theorem \ref{Thm4}, so it follows from that theorem that $(q_0,\ldots,q_k)\in\mathcal V(\theta_{n,k})$. 
\end{proof}

In the proof of Theorem \ref{Thm10}, we showed that $\mathcal V(\theta_{n,k})=\Comp_{k+1}(n-k)$. According to Theorem \ref{Thm7}, the number of permutations in $s^{-1}(\theta_{n,k})$ with $m$ peaks is \[\sum_{(q_0,\ldots,q_k)\in\Comp_{k+1}(n-k)}\sum_{(j_0,\ldots,j_k)\in\Comp_{k+1}(m+1)}\prod_{t=0}^k V(q_t,j_t).\] We leave it as an open problem to find a simple closed-form expression for these numbers. 

The following theorem gives a new combinatorial interpretation for the Fine numbers $F_n$, which are defined by \[\sum_{n\geq 0}F_nx^n=\frac{1}{x}\frac{1-\sqrt{1-4x}}{3-\sqrt{1-4x}}.\] The sequence of Fine numbers, which is sequence A000957 in the Online Encyclopedia of Integer Sequences \cite{OEIS}, has many combinatorial connections with Catalan numbers. See the survey \cite{Deutsch} for more on this ubiquitous sequence. The following theorem also involves two reFinements of the Fine numbers. These are the numbers \[g_{n,m}=\sum_{k=0}^{\left\lfloor\frac{n-1}{2}\right\rfloor}N_k(n-k,m-k+1)\] and \[h_{n,m}=\frac{2^{n-2m-1}}{n+2}{n+2\choose m+1}{n-m-1\choose m}.\] These numbers, which have combinatorial interpretations in terms of Dyck paths, appear as sequences A100754 and A114593 in \cite{OEIS}. It is known that\footnote{These identities are stated without proof in the Online Encyclopedia of Integer Sequences, but they can be proven by standard (yet somewhat tedious) arguments involving generating functions.} 
\begin{equation}\label{Eq8}
\sum_{m=0}^{n-1}g_{n,m}=\sum_{m=0}^{\left\lfloor\frac{n-1}{2}\right\rfloor}h_{n,m}=F_{n+1}.
\end{equation}

\begin{theorem}\label{Thm11}
In the notation of the preceding paragraph, we have \[|s^{-1}(\Av_n(132,231,312))|=F_{n+1}.\] Moreover, the number of permutation in $s^{-1}(\Av_n(132,231,312))$ with exactly $m$ descents is $g_{n,m}$. The number of permutation in $s^{-1}(\Av_n(132,231,312))$ with exactly $m$ peaks is $h_{n,m}$.
\end{theorem}

\begin{proof}
Recall that $\Av_n(132,231,312)=\{\theta_{n,0},\theta_{n,1},\ldots,\theta_{n,n-1}\}$. It follows from Theorem \ref{Thm10} that $s^{-1}(\theta_{n,k})$ is empty if $k>\left\lfloor\frac{n-1}{2}\right\rfloor$. It now follows immediately from Theorem \ref{Thm5} that the number of permutations in $s^{-1}(\Av_n(132,231,312))$ with exactly $m$ descents is $g_{n,m}$. Along with \eqref{Eq8}, this implies that $|s^{-1}(\Av_n(132,231,312))|=F_{n+1}$.

Recall the generating function $G(x,y)$ from \eqref{Eq4}. According to the preceding paragraph and Theorem \ref{Thm7}, the number of permutations in $s^{-1}(\theta_{n,k})$ with $m$ peaks is \[\sum_{(q_0,\ldots,q_k)\in\Comp_{k+1}(n-k)}\,\,\sum_{(j_0,\ldots,j_k)\in\Comp_{k+1}(m+1)}\,\,\prod_{t=0}^kV(q_t,j_t).\] This is nothing more than the coefficient of $x^ny^m$ in $x^ky^k\,G(x,y)^{k+1}$. Consequently, the number of permutations in $s^{-1}(\Av_n(132,231,312))$ with $m$ peaks is \[[x^ny^m]\left(\sum_{k=0}^{n-1}x^ky^k\,G(x,y)^{k+1}\right)=[x^ny^m]\left(\sum_{k=0}^\infty x^ky^k\,G(x,y)^{k+1}\right)=[x^ny^m]\left(\frac{G(x,y)}{1-xy\,G(x,y)}\right).\]
Straightforward algebraic manipulations show that \[\frac{G(x,y)}{1-xy\,G(x,y)}=\frac{1-2x-\sqrt{1-4x+4x^2-4x^2y}}{xy(1+2x+\sqrt{1-4x+4x^2-4x^2y})}.\] This last expression is equal to the generating function $\displaystyle\sum_{n\geq 1}\sum_{m=0}^{\left\lfloor\frac{n-1}{2}\right\rfloor}h(n,m)x^ny^m$.  
\end{proof}

\section{$s^{-1}(\Av(312,321))$}\label{Sec:312,321}
Applying the algorithm of Claesson and \'Ulfarsson shows that $s^{-1}(\Av(312,321))$ is equal to the set of permutations avoiding certain mesh patterns. However, Chetak Hossain has drawn the current author's attention to the fact that this set is actually a permutation class. Specifically, $s^{-1}(\Av(312,321))=\Av(3412,3421)$. We leave the proof of this statement to the reader; it requires nothing more than the definition of the stack-sorting map and the definition of permutation pattern avoidance. Hossain has also pointed the current author to the paper \cite{Kremer}. In this paper, Kremer proves that $|\Av_n(3412,3421)|$ is the $(n-1)^\text{th}$ large Schr\"oder number. In other words, we have the following theorem.\footnote{Kremer did not mention the stack-sorting map in her theorem.} 

\begin{theorem}[\!\!\cite{Kremer}]\label{Thm17}
We have \[\sum_{n\geq 1}|s^{-1}(\Av_n(312,321))|x^n=\sum_{n\geq 1}|\Av_n(3412,3421)|x^n=\frac{1-x-\sqrt{1-6x+x^2}}{2}.\] 
\end{theorem}
 
\section{$s^{-1}(\Av(132,321))$}\label{Sec:132,321} 

We have \[s^{-1}(\Av(132,321))=\Av(1342, 34251, 35241, 45231, \underline{31}42).\] 
It appears as though the sequence enumerating these permutations has not been studied before, but we will see that its generating function is fairly simple. Let $C(x)=\sum_{n\geq 0}C_nx^n=\dfrac{1-\sqrt{1-4x}}{2x}$ be the generating function of the sequence of Catalan numbers. Recall the generating functions $F(x,y)$ and $G(x,y)$ from \eqref{Eq2} and \eqref{Eq4}. Let $\mathfrak a(n,m)$ denote the number of elements of $s^{-1}(\Av_n(132,321))$ with exactly $m$ descents, and let $\mathfrak b(n,m)$ denote the number of elements of $s^{-1}(\Av_n(132,321))$ with exactly $m$ peaks.  

\begin{theorem}
In the notation of the preceding paragraph, we have \[\sum_{n\geq 1}|s^{-1}(\Av_n(132,321))|\,x^n=C(x)-1+x^3(C'(x))^2.\] Furthermore, \begin{equation}\label{Eq12}
\sum_{n\geq 1}\sum_{m\geq 0}\mathfrak a(n,m)x^ny^m=F(x,y)+x^3y\left(\frac{\partial}{\partial x}F(x,y)\right)^2,
\end{equation} and 
\begin{equation}\label{Eq13}
\sum_{n\geq 1}\sum_{m\geq 0}\mathfrak b(n,m)x^ny^m=G(x,y)+x^3y\left(\frac{\partial}{\partial x}G(x,y)\right)^2.
\end{equation}
\end{theorem} 

\begin{proof}
We will prove \eqref{Eq12}; the proof of \eqref{Eq13} is similar. In addition, the first statement of the theorem follows from \eqref{Eq12} and the fact that $F(x,1)=C(x)-1$. 

For $h,i,t\geq 1$, let \[\delta_{h,i,t}=(h+1)(h+2)\cdots(h+i)12\cdots h(h+i+1)(h+i+2)\cdots(h+i+t)\in S_{h+i+t}.\] For example, $\delta_{1,3,2}=234156$. It is straightforward to check that \[\Av_n(132,321)=\{123\cdots n\}\cup\{\delta_{h,i,t}:h,i,t\geq 1,\,\,h+i+t=n\}.\] Moreover, the set of valid compositions of $\delta_{h,i,t}$ is \[\mathcal V(\delta_{h,i,t})=\{(i+j-\ell,h+\ell-1):1\leq \ell\leq t\}.\] Invoking Theorem \ref{Thm7}, we find that \[\mathfrak a(n,m)=N(n,m+1)+\sum_{\substack{h,i,t\geq 1\\ h+i+t=n}}\sum_{\ell=1}^t\sum_{j=1}^mN(i+t-\ell,j)N(h+\ell-1,m+1-j)\] \[=N(n,m+1)+\sum_{h=1}^{n-2}\sum_{i=1}^{n-h}\sum_{\ell=1}^{n-h-i}\sum_{j=1}^mN(n-h-\ell,j)N(h+\ell-1,m+1-j)\] \[=N(n,m+1)+\sum_{h=1}^{n-2}\sum_{\ell=1}^{n-h}\sum_{i=1}^{n-h-\ell}\sum_{j=1}^mN(n-h-\ell,j)N(h+\ell-1,m+1-j)\] \[=N(n,m+1)+\sum_{h=1}^{n-2}\sum_{\ell=1}^{n-h}\sum_{j=1}^m(n-h-\ell)N(n-h-\ell,j)N(h+\ell-1,m+1-j).\] The substitution $r=n-h-\ell$ gives \[\mathfrak a(n,m)=N(n,m+1)+\sum_{h=1}^{n-2}\sum_{r=1}^{n-h-1}\sum_{j=1}^mr\,N(r,j)N(n-r-1,m+1-j)\] \[=N(n,m+1)+\sum_{r=1}^{n-2}\sum_{j=1}^mr(n-r-1)\,N(r,j)N(n-r-1,m+1-j).\] It is now routine to verify that this last expression is the coefficient of $x^ny^m$ in \[F(x,y)+x^3y\left(\frac{\partial}{\partial x}F(x,y)\right)^2.\qedhere\] 
\end{proof}

\section{$s^{-1}(\Av(132,312))$ and $s^{-1}(\Av(231,312))$}\label{Sec:Pair}

We have \[s^{-1}(\Av(132,312))=\Av(1342,3142,3412,34\underline{21})\] and \[s^{-1}(\Av(231,312))=\Av(2341,3412,34\underline{21},3\overline{5}241).\] 
From these descriptions of these sets, there is no obvious reason to expect that $|s^{-1}(\Av_n(132,312))|$ 
$=|s^{-1}(\Av_n(231,312))|$. However, this is indeed the case; valid hook configurations make the proof quite painless.  

\begin{theorem}\label{Thm12}
For all positive integers $n$, we have \[|s^{-1}(\Av_n(132,312))|=|s^{-1}(\Av_n(231,312))|.\] In fact, the number of permutations in $s^{-1}(\Av_n(132,312))$ with $m$ descents is the same as the number of permutations in $s^{-1}(\Av_n(231,312))$ with $m$ descents. Moreover, the number of permutations in $s^{-1}(\Av_n(132,312))$ with $m$ peaks is the same as the number of permutations in $s^{-1}(\Av_n(231,312))$ with $m$ peaks.
\end{theorem}

\begin{proof}
Invoking Theorem \ref{Thm5} and Theorem \ref{Thm7}, we see that it suffices to exhibit a bijection $\varphi:\Av_n(132,312)\to \Av_n(231,312)$ with the property that $\mathcal V(\pi)=\mathcal V(\varphi(\pi))$ for all $\pi\in\Av_n(132,312)$. 

Suppose we are given a permutation $\pi\in\Av_n(132,312)$. Recall that a \emph{left-to-right maximum} of a permutation is an entry of the permutation that is larger than all entries to its left. Because $\pi$ avoids $312$, we can write \[\pi=b_1(1)b_1(2)\cdots b_1(c_1)\,\,b_2(1)b_2(2)\cdots b_2(c_2)\,\,\cdots\,\, b_\ell(1)b_\ell(2)\cdots b_\ell(c_\ell),\] where the entries $b_i(1)$ are the left-to-right maxima of $\pi$ and the strings $b_i(1)b_i(2)\cdots b_i(c_i)$ are decreasing. In other words, the strings $b_i(1)b_i(2)\cdots b_i(c_i)$ are the descending runs of $\pi$. It turns out that $\pi$ is uniquely determined by the numbers $c_1,\ldots,c_\ell$ and the assumption that $\pi\in\Av_n(132,312)$. Indeed, the entries of $\pi$ that are not left-to-right maxima must appear in decreasing order lest $\pi$ contain a $312$ pattern or a $132$ pattern. These entries must also all be less than the first entry of the permutation. Each permutation in $\Av_n(231,312)$ is also uniquely determined by the lengths of its descending runs. This is because the entries in each descending run of the permutation must be consecutive integers. Indeed, the permutations in $\Av(231,312)$ are called \emph{layered permutations}; as described at the beginning of Section \ref{Sec6}, a permutation is in $\Av(231,312)$ if and only if it can be written as a sum of decreasing permutations. 

For example, the unique permutation in $\Av_{10}(132,312)$ whose descending runs have lengths $2,3,1,2,1,1$ is $5\,\,4\,\,6\,\,3\,\,2\,\,7\,\,8\,\,1\,\,9\,\,10$ (whose plot is shown on the left in Figure \ref{Fig6}). The unique permutation in $\Av_{10}(231,312)$ whose descending runs have lengths $2,3,1,2,1,1$ is $2\,\,1\,\,5\,\,4\,\,3\,\,6\,\,8\,\,7\,\,9\,\,10$ (whose plot is shown on the right in Figure \ref{Fig6}). We now define $\varphi:\Av_n(132,312)\to \Av_n(231,312)$ by declaring $\varphi(\pi)$ to be the unique permutation in $\Av_n(231,312)$ whose $i^\text{th}$ descending run has the same length as the $i^\text{th}$ descending run of $\pi$ for all $i$. 

Figure \ref{Fig6} illustrates the map $\varphi$. In this figure, we have drawn a valid hook configuration on each of the plots. In general, the valid hook configurations of $\pi\in\Av_n(132,312)$ correspond bijectively to the valid hook configurations of $\varphi(\pi)$. Specifically, each hook of a valid hook configuration of $\varphi(\pi)$ is obtained from a hook in the corresponding valid hook configuration of $\pi$ by keeping fixed the horizontal coordinates of the endpoints of the hooks. In other words, we can obtain the plot of $\varphi(\pi)$ by vertically sliding some of the points in the plot of $\pi$; we keep the hooks attached to their endpoints throughout this sliding motion. Corresponding valid hook configurations induce the same valid compositions, so we have $\mathcal V(\pi)=\mathcal V(\varphi(\pi))$ for all $\pi\in\Av_n(132,312)$.  
\end{proof}

\begin{figure}[t]
\begin{center}
\includegraphics[width=.6\linewidth]{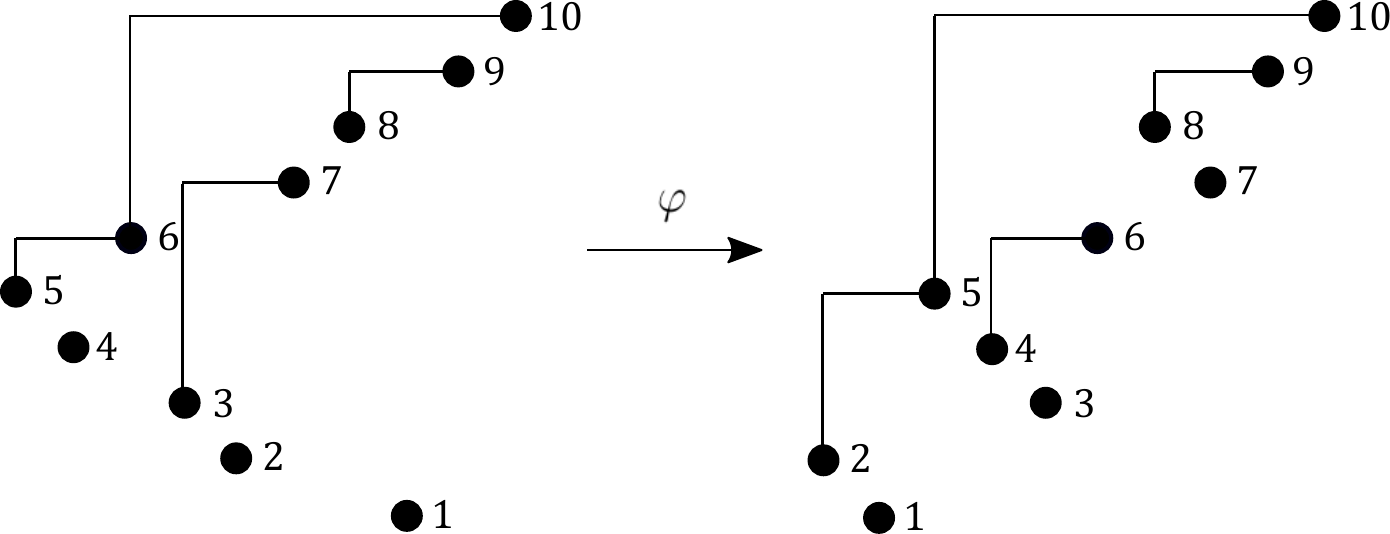}
\end{center}  
\caption{The map $\varphi$ from the proof of Theorem \ref{Thm12} sends the permutation $5\,\,4\,\,6\,\,3\,\,2\,\,7\,\,8\,\,1\,\,9\,\,10$ to the permutation $2\,\,1\,\,5\,\,4\,\,3\,\,6\,\,8\,\,7\,\,9\,\,10$. The valid hook configuration drawn on the plot of $5\,\,4\,\,6\,\,3\,\,2\,\,7\,\,8\,\,1\,\,9\,\,10$ corresponds to the one drawn on the plot of $2\,\,1\,\,5\,\,4\,\,3\,\,6\,\,8\,\,7\,\,9\,\,10$. Both valid hook configurations induce the valid composition $(1,1,2,1,1)$.}\label{Fig6}
\end{figure}

We have seen that the sequences $(|s^{-1}(\Av_n(132,312))|)_{n\geq 1}$ and $(|s^{-1}(\Av_n(231,312))|)_{n\geq 1}$ are identical. Numerical evidence suggests that this sequence is, up to reindexing, the same as the sequence A071356 in the Online Encyclopedia of Integer Sequences \cite{OEIS}. The latter sequence is defined as the expansion of a relatively simple generating function, but it also has some combinatorial interpretations. In addition, it appears as though $(|s^{-1}(\Av_n(132,231))|)_{n\geq 1}$ is the same sequence. We state these observations formally in the following conjecture.  

\begin{conjecture}\label{Conj2}
We have \[\sum_{n\geq 1}|s^{-1}(\Av_n(132,312))|x^n=\sum_{n\geq 1}|s^{-1}(\Av_n(132,231))|x^n=\frac{1-2x-\sqrt{1-4x-4x^2}}{4x}.\]
\end{conjecture}

\section{$s^{-1}(\Av(321))$}\label{Sec:321}

As discussed in the introduction, there are known formulas for $|s^{-1}(\Av_n(\tau))|$ whenever $\tau$ is a permutation pattern of length $3$ other than $321$. By contrast, the sequence $(|s^{-1}(\Av_n(321))|)_{n\geq 1}$ appears to be new.\footnote{We have added it as sequence A319027 in the Online Encyclopedia of Integer Sequences.} This sequence is of interest because $s^{-1}(\Av(321))$ is equal to the permutation class $\Av(34251,35241,45231)$. We will use valid hook configurations to establish nontrivial estimates for the growth rate of this sequence. Note that the trivial estimates for this growth rate\footnote{The nitpicky reader might beg for a proof of the existence of the limit defining this growth rate. We say a sequence of real numbers $(a_m)_{m=1}^\infty$ is \emph{supermultiplicative} if $a_ma_n\leq a_{m+n}$ for all positive integers $m,n$. The multiplicative version of Fekete's lemma  \cite{Fekete} states that if $(a_m)_{m=1}^\infty$ is a supermultiplicative sequence, then $\displaystyle{\lim_{m\to\infty}\sqrt[m]{a_m}}$ exists. It is straightforward to show (in the notation of Section \ref{Sec6}) that $s(\sigma\oplus\mu)=s(\sigma)\oplus s(\mu)$ for any $\sigma\in S_m$ and $\mu\in S_n$. It follows that there is an injective map $s^{-1}(\Av_m(321))\times s^{-1}(\Av_n(321))\to s^{-1}(\Av_{m+n}(321))$ given by $(\sigma,\mu)\mapsto\sigma\oplus\mu$. Hence, $(|s^{-1}(\Av_n(321))|)_{n\geq 1}$ is supermultiplicative.} are given by \[4\leq\lim_{n\to\infty}|s^{-1}(\Av_n(321))|^{1/n}\leq 16.\] The lower bound follows from the fact  that $|s^{-1}(123\cdots n)|=C_n$ has growth rate $4$. We know the upper bound because $|\Av_n(321)|=C_n$ has growth rate $4$ and because each permutation of length $n$ has fertility at most $4^n$.\footnote{From a permutation $\sigma\in S_n$, we obtain a word $w_\sigma$ of length $2n$ over the alphabet $\{A,B\}$ as follows. Send $\sigma$ through the stack as described in the introduction. Each time an entry is pushed into the stack, write an $A$. Each time an entry is popped out of the stack, write a $B$. For example, the permutation $3142$ from Figure \ref{Fig1} gives rise to the word $w_{3142}=AABBAABB$. It is straightforward to check that, for each $\pi\in S_n$, the map $\sigma\mapsto w_\sigma$ is injective on $s^{-1}(\pi)$. Therefore, $|s^{-1}(\pi)|\leq 2^{2n}=4^n$.} 

\begin{theorem}\label{Thm13}
We have \[8.4199\leq\lim_{n\to\infty}|s^{-1}(\Av_n(321))|^{1/n}\leq 11.6569.\]
\end{theorem}

\begin{proof}
By reversing permutations, we see that $|\Av_{n,k}(321)|=|\Av_{n,n-1-k}(123)|$. The authors of \cite{Barnabei} have computed $|\Av_{n,k}(123)|$; one can easily use their results to see that the values of $|\Av_{n,k}(321)|$ are given by sequence A091156 in the Online Encyclopedia of Integer Sequences \cite{OEIS} (this interpretation of the sequence A091156 is not new to this article). A precise formula is given by 
\begin{equation}\label{Eq14}
|\Av_{n,k}(321)|=\frac{1}{n+1}{n+1\choose k}\sum_{j=0}^{n+1-2k}{k+j-1\choose k-1}{n+1-k\choose n-2k-j}.
\end{equation}

To prove the desired upper bound, we combine Theorem \ref{Thm14} with \eqref{Eq14} to find that 
\[|s^{-1}(\Av_n(321))|=\sum_{k=0}^{n-1}|s^{-1}(\Av_{n,k}(321))|\leq\sum_{k=0}^{n-1}\frac{2k+2}{n+1}{2n-2k-1\choose n}|\Av_{n,k}(321)|\] \[=\sum_{k=0}^{n-1}\frac{2k+2}{n+1}{2n-2k-1\choose n}\frac{1}{n+1}{n+1\choose k}\sum_{j=0}^{n+1-2k}{k+j-1\choose k-1}{n+1-k\choose n-2k-j}.\] Up to a subexponential factor, this upper bound is \[\sum_{k=0}^{n-1}{2n-2k\choose n}{n\choose k}\sum_{j=0}^{n-2k}{k+j\choose k}{n-k\choose n-2k-j}.\] The sum over $j$ in the last expression is maximized when $j\sim n/2-k$, so our upper bound is (again, up to a subexponential factor) \[\sum_{k=0}^{n-1}{2n-2k\choose n}{n\choose k}\sum_{j=0}^{n-2k}{n/2\choose k}{n-k\choose n/2-k}=\sum_{k=0}^{n-1}(n-2k+1){2n-2k\choose n}{n\choose k}{n/2\choose k}{n-k\choose n/2-k}.\] Let $K(n)$ denote the value of $k$ for which the term in this last summation is maximized, and put $c(n)=K(n)/n$. Note that $c(n)\in(0,1/2)$. A straightforward application of Stirling's formula shows that, up to a subexponential factor, this last upper bound is at most \[\frac{(2-2c(n))^{2-2c(n)}}{(1-2c(n))^{1-2c(n)}}\frac{1}{c(n)^{c(n)}(1-c(n))^{1-c(n)}}\frac{(1/2)^{1/2}}{c(n)^{c(n)}(1/2-c(n))^{1/2-c(n)}}\frac{(1-c(n))^{1-c(n)}}{(1/2-c(n))^{1/2-c(n)}(1/2)^{1/2}}\] \[=f(c(n)),\] where \[f(x)=\frac{(2-2x)^{2-2x}}{(1-2x)^{1-2x}x^{2x}(1/2-x)^{1-2x}}.\] One can easily verify that $f(x)\leq 11.6569$ whenever $x\in(0,1/2)$. This proves the desired upper bound.    

The \emph{reverse complement} of a permutation $\pi_1\cdots\pi_n\in S_n$ is the permutation whose $i^\text{th}$ entry is $n+1-\pi_{n+1-i}$. The proof of the desired lower bound requires two crucial observations. The first is that the set $\Av_{n,k}(321)$ is closed under taking reverse complements. If $\pi\in S_n$ has $\ell$ left-to-right maxima, then the reverse complement of $\pi$ has $n-\ell$ left-to-right maxima. It follows that at least half of the permutations in $\Av_{n,k}(321)$ have at least $n/2$ left-to-right maxima. 

The second observation is that if $\pi\in \Av_{n,k}(321)$ has last entry $\pi_n=n$, then $\pi$ has a valid hook configuration (i.e., $\pi$ is sorted). Indeed, let $d_1<\cdots<d_k$ denote the descents of $\pi$. For each $i\in\{1,\ldots,k\}$, let $\pi_{b_i}$ be the leftmost left-to-right maximum of $\pi$ that lies to the right of $\pi_{d_i}$. The condition $\pi_n=n$ guarantees that this left-to-right maximum exists for all $i$. The canonical valid hook configuration $(H_1^*,\ldots,H_k^*)$ of $\pi$ (described in Section \ref{Sec:VHCs}) is formed by declaring that $H_i^*$ has southwest endpoint $(d_i,\pi_{d_i})$ and northeast endpoint $(b_i,\pi_{b_i})$ for all $i$. 

There are exactly $|\Av_{n-1,k}(321)|$ permutations in $\Av_{n,k}(321)$ with last entry $n$ (we obtain a bijection by adding the entry $n$ to the end of a permutation in $\Av_{n-1,k}(321)$). According to the discussion above, there are at least $\frac 12|\Av_{n-1,k}(321)|$ permutations in $\Av_{n,k}(321)$ with at least $n/2$ left-to-right maxima. Choose one such permutation $\pi$, and let $\ell$ be the number of left-to-right maxima in $\pi$. The canonical valid hook configuration of $\pi$ (described in the previous paragraph) induces a valid composition $(q_0^*,\ldots,q_k^*)\in\mathcal V(\pi)$. In the coloring of the plot of $\pi$ induced by the canonical valid hook configuration, there are exactly $\ell-k$ points colored blue (sky-colored). Indeed, these points are precisely the left-to-right maxima of $\pi$ that are not northeast endpoints of hooks. This tells us that $q_0^*=\ell-k$. According to Theorem \ref{Thm5}, 
\begin{equation}\label{Eq17}
|s^{-1}(\pi)|\geq C_{(q_0^*,\ldots,q_k^*)}=C_{\ell-k}C_{q_1^*}\cdots C_{q_k^*}.
\end{equation} 

For convenience, we define $C_x=\dfrac{\Gamma(2x+1)}{\Gamma(x+2)\Gamma(x+1)}$, where $\Gamma$ denotes the Gamma function. When $x$ is a positive integer, $C_x$ is simply the $x^\text{th}$ Catalan number. One can show that 
\begin{equation}\label{Eq15}
C_{x+\varepsilon}C_{y-\varepsilon}<C_xC_y\quad\text{whenever}\quad 0<x<y\quad\text{and}\quad 0<\varepsilon\leq\frac{y-x}{2}.
\end{equation} In other words, a product of (generalized) Catalan numbers decreases when we make the indices closer while preserving the sum of the indices. Let us assume that $n$ is sufficiently large, that $5<k<0.4n$, and that $\pi$ is chosen as in the previous paragraph. By the properties of valid compositions, we know that 
\begin{equation}\label{Eq16}
\ell-k+q_1^*+\cdots+q_k^*=n-k.
\end{equation} It follows from \eqref{Eq15} and \eqref{Eq16} that 
\begin{equation}\label{Eq18}
C_{\ell-k}C_{q_1^*}\cdots C_{q_k^*}\geq C_{\ell-k}C_{(n-\ell)/k}^k.
\end{equation} The assumption $5<k<0.4n$ and the fact that $\ell\geq n/2$ guarantee that \[\frac{n-\ell}{k}\leq \frac{n}{2k}<\frac n2-k\leq\ell-k.\] By \eqref{Eq15}, we have 
\begin{equation}\label{Eq19}
C_{\ell-k}C_{(n-\ell)/k}^k\geq C_{n/2-k}C_{n/(2k)}^k.
\end{equation}

When we combine \eqref{Eq15}, \eqref{Eq18}, and \eqref{Eq19} with the discussion above, we find that there are at least $\frac 12|\Av_{n-1,k}(321)|$ permutations in $\Av_{n,k}(321)$ that each have at least $C_{n/2-k}C_{n/(2k)}^k$ preimages under $s$ (for $5<k<0.4n$). We now use \eqref{Eq14} to see that \[|s^{-1}(\Av_n(321))|\geq \frac 12|\Av_{n-1,k}(321)|C_{n/2-k}C_{n/(2k)}^k\] \[\geq\frac{1}{2n}{n\choose k}\sum_{j=0}^{n-2k}{k+j-1\choose k-1}{n-k\choose n-2k-j-1}C_{n/2-k}C_{n/(2k)}^k\] \[\geq\frac{1}{2n}{n\choose k}{k+(\left\lfloor n/2\right\rfloor-k)-1\choose k-1}{n-k\choose n-2k-(\left\lfloor n/2\right\rfloor-k)-1}C_{n/2-k}C_{n/(2k)}^k\] \[=\frac{1}{2n}{n\choose k}{\left\lfloor n/2\right\rfloor-1\choose k-1}{n-k\choose \left\lceil n/2\right\rceil-k-1}C_{n/2-k}C_{n/(2k)}^k.\] This holds whenever $5<k<0.4n$. In particular, we can put $k=\left\lfloor 0.06582n\right\rfloor$ (this value is chosen to maximize the lower bound). With this choice of $k$, we can use Stirling's formula to see that our lower bound is at least $8.4199^n$ for sufficiently large $n$.       
\end{proof}

\section{Concluding Remarks and Further Directions}

Let us collect some open problems and conjectures arising from and related to the topics studied in this article. 

Recall that a sequence $a_1,\ldots,a_m$ is called \emph{unimodal} if there exists $j\in\{1,\ldots,m\}$ such that $a_1\leq\cdots\leq a_{j-1}\leq a_j\geq a_{j+1}\geq\cdots\geq a_m$ and is called \emph{log-concave} if $a_j^2\geq a_{j-1}a_{j+1}$ for all $j\in\{2,\ldots,m-1\}$ \cite{Branden2}. This sequence is called \emph{real-rooted} if all of the complex roots of the polynomial $\sum_{k=1}^ma_kx^k$ are real. It is well-known that a real-rooted sequence of nonnegative numbers is log-concave and that a log-concave sequence of nonnegative numbers is unimodal.

The notions of unimodality, log-concavity, and real-rootedness are prominent in the study of the stack-sorting map. For example, let $f_k(\pi)$ denote the number of elements of $s^{-1}(\pi)$ with $k$ descents. B\'ona proved \cite{BonaSymmetry} that the sequence $f_0(\pi),\ldots,f_{n-1}(\pi)$ is symmetric and unimodal for each $\pi\in S_n$. We conjecture the following much stronger result, which we have verified for all permutations of length at most $8$. 

\begin{conjecture}\label{Conj3}
For each permutation $\pi\in S_n$, the sequence $f_0(\pi),\ldots,f_{n-1}(\pi)$ is real-rooted. 
\end{conjecture}

Of course, even if it is too difficult to prove that $f_0(\pi),\ldots,f_{n-1}(\pi)$ is always real-rooted, it would be very interesting to prove the weaker statement that this sequence is always log-concave. One could also attempt to find large classes of permutations for which Conjecture \ref{Conj3} holds. 

A consequence of B\'ona's result is that
\begin{equation}\label{Eq20}
W_t(n,0),\ldots,W_t(n,n-1)
\end{equation} is symmetric and unimodal for all $t,n\geq 1$, where $W_t(n,k)$ is the number of $t$-stack-sortable permutations in $S_n$ with $k$ descents. Knowing that the sequence in \eqref{Eq20} is real-rooted when $t=1$ and when $t=n$, B\'ona \cite{BonaSymmetry} conjectured that the sequence is real-rooted in general. Br\"and\'en \cite{Branden3} later proved this conjecture in the cases $t=2$ and $t=n-2$. This leads us to the following much more general problem. 

\begin{question}\label{Quest1}
Given a set $\mathcal U$ of permutations, let $f_k(\mathcal U\cap S_n)$ denote the number of permutations in $s^{-1}(\mathcal U\cap S_n)$ with exactly $k$ descents. Can we find interesting examples of sets $\mathcal U$ (such as permutation classes) with the property that $f_0(\mathcal U\cap S_n),\ldots,f_{n-1}(\mathcal U\cap S_n)$ is a real-rooted sequence for every $n\geq 1$? Is this sequence always real-rooted? 
\end{question} 

Recall from Section \ref{Sec:Back} that $s^{-1}(\Av_n(123))$ is empty when $n\geq 4$. In general, $s^{-1}(\Av_n(123\cdots m))$ is empty if $n\geq 2^{m-1}$. This is certainly true for $m\leq 3$. To see that this is true for $m\geq 4$, suppose $n\geq 2^{m-1}$ and $\pi\in S_n$. Write $\pi=LnR$ so that $s(\pi)=s(L)s(R)n$. One of $L$ and $R$ has length at least $2^{m-2}$, so it follows by induction on $m$ that either $s(L)$ or $s(R)$ contains an increasing subsequence of length $m-1$. Therefore, $s(\pi)$ contains the pattern $123\cdots m$. 

Despite the uninteresting behavior of the sequence $(|s^{-1}(\Av_n(123\cdots m))|)_{n\geq 1}$ for large values of $n$, it could still be interesting to study the initial terms in this sequence. For example, the nonzero terms of the sequence are $1,2,6,10,13,10,3$ when $m=4$. 

\begin{conjecture}\label{Conj4}
For each integer $m\geq 2$, the sequence $(|s^{-1}(\Av_n(123\cdots m))|)_{n=1}^{2^{m-1}-1}$ is unimodal. 
\end{conjecture}

We next consider some natural questions that we have not attempted to answer. 

\begin{question}\label{Quest6}
Can we obtain interesting results by enumerating sets of the form 
\linebreak $s^{-1}(\Av(\tau^{(1)},\ldots,\tau^{(r)}))$ when the patterns $\tau^{(1)},\ldots,\tau^{(r)}$ are not all of length $3$?
\end{question} 

\begin{question}\label{Quest2}
Let $S_{n,k}$ denote the set of permutations in $S_n$ with exactly $k$ descents. Can we find formulas for $|s^{-1}(S_{n,1})|$, $|s^{-1}(S_{n,2})|$, or $|s^{-1}(S_{n,3})|$? 
\end{question} 

Throughout this article, we often enumerated permutations in sets of the form \linebreak
$s^{-1}(\Av_n(\tau^{(1)},\ldots,\tau^{(r)}))$ according to their number of descents or number of peaks. This is because the techniques in \cite{Defant} allow us to use valid hook configurations to count preimages of permutations according to these statistics. It is likely that there are other permutation statistics that can be treated similarly using valid hook configurations. It would be interesting to obtain results for these other statistics analogous to those derived above for descents and peaks. 

The article \cite{Defant} provides a general method for computing the number of decreasing plane trees of various types that have a specified permutation as their postorder reading (see the article for the relevant definitions). In the special case in which the trees are decreasing binary plane trees, this is equivalent to computing fertilities of permutations. This suggests that one could obtain enumerative results analogous to those from this paper by replacing decreasing binary plane trees with other types of trees. This provides a very general new collection of enumerative problems. Namely, we want to count the decreasing plane trees of a certain type whose postorders lie in some permutation class. Two very specific examples of this type of problem are the following. Preserve the notation from the article \cite{Defant}. 

\begin{question}\label{Quest3}
How many decreasing $\mathbb N$-trees have postorders that lie in the set $\Av_n(132,231,312)$? How many unary-binary trees have postorders that lie in the set $\Av_n(132,231,312)$?
\end{question}

We now collect the open problems and conjectures that arose throughout Sections 4--10. First, recall the following conjecture from Section \ref{Sec6}. 

\begin{reptheorem}{Conj1}
In the notation of Theorem \ref{Thm8}, we have \[\sum_{q\,\in\,\Comp_{k+1}(n-k)}C_qD_q=\frac{1}{n+1}{n-k-1\choose k}{2n-2k\choose n}\] for all nonnegative integers $n$ and $k$. 
\end{reptheorem}

In Section \ref{Sec7}, we defined \[\theta_{n,k}=(k+1)k(k-1)\cdots 321(k+2)(k+3)\cdots n\] and mentioned that the number of permutations in $s^{-1}(\theta_{n,k})$ with $m$ peaks is 
\begin{equation}\label{Eq21}
\sum_{(q_0,\ldots,q_k)\in\Comp_{k+1}(n-k)}\sum_{(j_0,\ldots,j_k)\in\Comp_{k+1}(m+1)}\prod_{t=0}^k V(q_t,j_t).
\end{equation} 

\begin{question}\label{Quest4}
Can we find a simple closed form for the expression in \eqref{Eq21}?
\end{question}  

We stated the following intriguing conjecture in Section \ref{Sec:Pair}. 

\begin{reptheorem}{Conj2}
We have \[\sum_{n\geq 1}|s^{-1}(\Av_n(132,312))|x^n=\sum_{n\geq 1}|s^{-1}(\Av_n(132,231))|x^n=\frac{1-2x-\sqrt{1-4x-4x^2}}{4x}.\]
\end{reptheorem}

Of course, our results in Section \ref{Sec:321} are far from perfect. 

\begin{question}\label{Quest5}
Can we enumerate the permutations in $s^{-1}(\Av(321))$ exactly? Can we at least improve the estimates from Theorem \ref{Thm13}? 
\end{question}

We have focused primarily on preimage sets of the form $s^{-1}(\Av(\tau^{(1)},\ldots,\tau^{(r)}))$ when 
$\emptyset\neq\{\tau^{(1)},\ldots,\tau^{(r)}\}\subseteq\{132,231,312,321\}$. The astute reader may have realized that the only preimage set of this form that we have not mentioned is $s^{-1}(\Av(231,321))$. This set appears to be enumerated by the OEIS sequence A165543 \cite{OEIS}. More precisely, we have the following conjecture.  

\begin{conjecture}\label{Conj5}
We have \[\sum_{n\geq 0}|s^{-1}(\Av_n(231,321))|x^n=\frac{1}{1-xC(xC(x))},\] where $C(x)=\dfrac{1-\sqrt{1-4x}}{2x}$ is the generating function of the sequence of Catalan numbers. 
\end{conjecture}

\section{Acknowledgments}
The author thanks Amanda Burcroff for a useful conversation about skyhooks. He also thanks Michael Engen for valuable conversations and thanks Steve Butler for providing data that was used to formulate several theorems and conjectures. He thanks Chetak Hossain for the contributions mentioned in Section \ref{Sec:312,321}. The author was supported by a Fannie and John Hertz Foundation Fellowship and an NSF Graduate Research Fellowship.


\begin{thebibliography}{1}
\bibitem{Babson}
E. Babson and E. Steingr\'imsson, Generalized permutation patterns and a classification of the Mahonian statistics. \emph{S\'em. Lothar.
Combin.} {\bf 44} (2000), Art. B44b.

\bibitem{Barnabei}
M. Barnabei, F. Bonetti, and M. Silimbani, The descent statistic on $123$-avoiding permutations. \emph{S\'em. Lothar. Combin.}, {\bf 63} (2010). 

\bibitem{Bevan}
D. Bevan, R. Brignall, A. E. Price, and J. Pantone, Staircases, dominoes, and the growth rate of $1324$-avoiders. \emph{Electron. Notes Discrete Math.}, {\bf 61} (2017), 123--129. 

\bibitem{Bona}
M. B\'ona, Combinatorics of permutations. CRC Press, 2012. 

\bibitem{BonaSimplicial}
M. B\'ona, A simplicial complex of 2-stack sortable permutations. \emph{Adv. Appl. Math.}, {\bf 29} (2002), 499--508.

\bibitem{BonaSurvey}
M. B\'ona, A survey of stack-sorting disciplines. \emph{Electron. J. Combin.}, {\bf 9.2} (2003): 16.

\bibitem{BonaSymmetry}
M. B\'ona, Symmetry and unimodality in $t$-stack sortable permutations. \emph{J. Combin. Theory Ser.
A}, {\bf 98.1} (2002), 201–-209.

\bibitem{Bousquet02}
M. Bousquet-M\'elou, Counting walks in the quarter plane. \emph{Mathematics and Computer Science II}, (2002), 49--67.

\bibitem{Bousquet98}
M. Bousquet-M\'elou, Multi-statistic enumeration of two-stack sortable permutations. \emph{Electron. J. Combin.}, {\bf 5} (1998), \#R21.

\bibitem{Bousquet}
M. Bousquet-M\'elou, Sorted and/or sortable permutations. \emph{Discrete Math.}, {\bf 225} (2000), 25–-50.

\bibitem{Bousquet06}
M. Bousquet-M\'elou and S. Butler, Forest-like permutations. \emph{Ann. Comb.}, {\bf 11} (2007), 335--354.

\bibitem{Bouvel}
M. Bouvel and O. Guibert, Refined enumeration of permutations sorted with two stacks and a $D_8$-symmetry. \emph{Ann. Comb.}, {\bf 18} (2014), 199--232. 

\bibitem{Branden3}
P. Br\"and\'en, On linear transformations preserving the Pólya frequency property. \emph{Trans. Amer. Math. Soc.},  {\bf 358} (2006), 3697--3716.

\bibitem{Branden2}
P. Br\"and\'en, Unimodality, log-concavity, real-rootedness and beyond, in: Handbook of Enumerative Combinatorics. CRC Press, 2015.

\bibitem{Branden}
P. Br\"and\'en and A. Claesson, Mesh patterns and the expansion of permutation statistics as sums of permutation patterns. \emph{Electron. J. Combin.}, {\bf 18} (2011), \#P5.

\bibitem{Bruner}
M.-L. Bruner, Central binomial coefficients also count $(2431, 4231, 1432, 4132)$-avoiders. arXiv:1505.04929. 

\bibitem{Claesson}
A. Claesson and H. \'Ulfarsson, Sorting and preimages of pattern classes, arXiv:1203.2437. 

\bibitem{Cori}
R. Cori, B. Jacquard, and G. Schaeffer, Description trees for some families of planar maps, \emph{Proceedings of the 9th FPSAC}, (1997). 

\bibitem{Defant4}
C. Defant, Fertility numbers. arXiv:1809.04421. 

\bibitem{Defant}
C. Defant, Postorder preimages. \emph{Discrete Math. Theor. Comput. Sci.}, {\bf 19}; 1 (2017). 

\bibitem{Defant2}
C. Defant, Preimages under the stack-sorting algorithm. \emph{Graphs Combin.}, {\bf 33} (2017), 103--122. 

\bibitem{Defant3}
C. Defant, M. Engen, and J. A. Miller, Stack-sorting, set partitions, and Lassalle's sequence. arXiv:1809.01340.  

\bibitem{Deutsch}
E. Deutsch and L. Shapiro, A survey of the Fine numbers. \emph{Discrete Math.}, {\bf 241} (2001), 241--265. 

\bibitem{Duchi}
E. Duchi, V. Guerrini, S. Rinaldi, and G. Schaeffer, Fighting fish. \emph{J. Phys. A.}, {\bf 50} (2017). 

\bibitem{Dulucq}
S. Dulucq, S. Gire, and O. Guibert, A combinatorial proof of J. West's conjecture. \emph{Discrete Math.}, {\bf 187} (1998), 71--96.

\bibitem{Dulucq2}
S. Dulucq, S. Gire, and J. West, Permutations with forbidden subsequences and nonseparable planar maps, \emph{Discrete Math.}, {\bf 153.1} (1996), 85--103.

\bibitem{Fang}
W. Fang, Fighting fish and two-stack-sortable permutations, arXiv:1711.05713. 

\bibitem{Fekete}
M. Fekete, \"Uber die Verteilung der Wurzeln bei gewissen algebraischen Gleichungen mit. ganzzahligen Koeffizienten, \emph{Math. Zeit.} {\bf 17} (1923), 228--249.

\bibitem{Goulden}
I. Goulden and J. West, Raney paths and a combinatorial relationship between rooted nonseparable planar maps and two-stack-sortable permutations. \emph{J. Combin. Theory Ser. A.}, {\bf 75.2} (1996), 220--242.

\bibitem{Knuth}
D. E. Knuth, The Art of Computer Programming, volume 1, Fundamental Algorithms.
Addison-Wesley, Reading, Massachusetts, 1973.

\bibitem{Kremer}
D. Kremer, Permutations with forbidden subsequences and a generalized Schr\"oder number. \emph{Discrete Math.}, {\bf 218} (2000), 121--130. 

\bibitem{Lassalle}
M. Lassalle, Two integer sequences related to Catalan numbers. \emph{J. Combin. Theory Ser. A}, {\bf 119} (2012), 923--935. 

\bibitem{OEIS}
The On-Line Encyclopedia of Integer Sequences, published electronically at http://oeis.org, 2010.

\bibitem{Stanley}
R. P. Stanley, Algebraic combinatorics. Springer, New York, 2013.  

\bibitem{Steingrimsson}
E. Steingr\'imsson, Generalized permutation patterns---a short survey. In \emph{Permutation Patterns} (2010), S. Linton, N. Ru\u{s}kuc, and V. Vatter (eds.), vol. 376 of \emph{London Math.
Soc. Lecture Note Ser.}, 137--152. 

\bibitem{West}
J. West, Permutations with restricted subsequences and stack-sortable permutations, Ph.D. Thesis, MIT, 1990.

\bibitem{Zeilberger}
D. Zeilberger, A proof of Julian West's conjecture that the number of two-stack-sortable permutations of length
$n$ is $2(3n)!/((n + 1)!(2n + 1)!)$. \emph{Discrete Math.}, {\bf 102} (1992), 85–-93.  
\end{thebibliography}
\end{document}